\documentclass[12pt,a4wide]{article}
\usepackage{amsmath, amsfonts, amssymb, amsthm, euscript, amscd, latexsym, mathrsfs, bm}
 \usepackage[pdftex,colorlinks=true,
                       pdfstartview=FitV,
                       linkcolor=blue,
                       citecolor=blue,
                       urlcolor=blue,
           ]{hyperref}
 \usepackage{hyperref}
 \usepackage{titlesec}
\def\eee#1{ \begin{equation} #1 \end{equation} }
\def\aa#1{ \begin{align*} #1 \end{align*} }
\def\aaa#1{ \begin{align} #1 \end{align} }
\def\mm#1{ \begin{multline*} #1 \end{multline*} }
\def\mmm#1{ \begin{multline} #1 \end{multline} }

\newtheorem{thm}{\sc Theorem}
\newtheorem{lem}{\sc Lemma}
\newtheorem{cor}{\sc Corollary}

\newtheorem{rem}{\sc Remark}
\newtheorem{asm}{\sc Assumption}

\newcommand{\ts}{\textstyle}
\newcommand{\sss}{\scriptscriptstyle}

\newcommand{\pl}{\partial}
\newcommand{\gt}{\geqslant}
\newcommand{\lt}{\leqslant}
\newcommand{\te}{\theta}
\newcommand{\sub}{\subset}

\newcommand{\dl}{\delta}
\newcommand{\al}{\alpha}
\newcommand{\gm}{\gamma}

 \newcommand{\sg}{\sigma}

\newcommand{\om}{\omega}
\newcommand{\mc}{\mathcal}
\newcommand{\Om}{\Omega}

\newcommand{\td}{\tilde}

\newcommand{\ox}{\otimes}

\newcommand{\we}{\wedge}

\renewcommand{\div}{\mathrm{div} \,}
\newcommand{\s}{{\alpha}}
\newcommand{\V}{\sss\mathrm{V}}
\newcommand{\x}{\times}
\newcommand{\mto}{\mapsto}

\newcommand{\E}{\mathbb E}
\newcommand{\PP}{\mathbb P}

\newcommand{\Gr}{{G^{\,\s}_{\V}}}
\newcommand{\Alg}{T_e\Gr}

\newcommand{\C}{{\rm C}}

\newcommand{\Proj}{{\rm P \,}}

\newcommand{\ovl}{\overline}

\DeclareMathOperator{\ind}{\mathbb I}
\newcommand{\lap}{\Delta}
\newcommand{\nab}{\nabla}
\newcommand{\fdot}{\,\cdot\,}
\def\Rnu{{\mathbb R}}

\def\Znu{{\mathbb Z}}
\def\Qnu{{\mathbb Q}}
\def\Tor{{\mathbb T}^2}

\def\ffi{\varphi}

\def\com#1{}

\long\def\symbolfootnote[#1]#2{\begingroup%
\def\thefootnote{\fnsymbol{footnote}}\footnote[#1]{#2}\endgroup}
\titleformat{\section}[hang]{\large\bfseries}{\thesection.}{1ex}{}{}
\titleformat{\subsection}[hang]{\normalsize\bfseries}{\thesubsection}{2ex}{}{}
\titleformat{\subsubsection}[hang]{\small\bfseries}{\thesubsubsection}{2ex}{}{}

\include{srctex.sty}

\begin{document}

 \author{Ana Bela Cruzeiro$^{1,2}$ and Evelina Shamarova$^{2,3}$
 }

 \date{}

 \title{ \Large
  \textbf{Navier--Stokes equations and forward-backward SDEs
   on the group of diffeomorphisms of a torus
  }
 }

 \maketitle

 \vspace{-11mm}

 {\small
  \begin{center}
  \begin{tabular}{l}
  { $^1$\! Dep. de Matem\'{a}tica, IST-UTL.}   \\
  { $^2$\! Grupo de F\'isica Matem\'atica da Universidade de Lisboa.} \\
  { $^3$\! Centro de Matem\'atica da Universidade do Porto.}  
  \end{tabular}
  \begin{tabular}{ll}
  { \hspace{2mm} E-mail addresses:}&{\href{mailto:abcruz@math.ist.utl.pt}{abcruz@math.ist.utl.pt} (A.B. Cruzeiro)}\\
 &{\href{mailto:evelina@cii.fc.ul.pt}{evelinas@fc.up.pt} (E. Shamarova)}
  \end{tabular}
 \end{center}
 }

 \vspace{3mm}

 \begin{abstract}
 We establish a connection between the strong solution
 to the spatially periodic Navier--Stokes equations
 and a solution to a system of forward-backward stochastic differential
 equations (FBSDEs) on the group of volume-preserving diffeomorphisms of
 a flat torus.
 We construct representations of the strong solution to
 the Navier--Stokes equations in terms of diffusion processes.
 \end{abstract}

\vspace{5mm}


\section{Introduction}
		The classical Navier--Stokes equations read as follows:
\aaa{
\label{ns-main}
\begin{split}
&\frac{\pl}{\pl t} u(t,x)
= - (u,\nab) u(t,x)
+ \nu \lap u(t,x) - \nab p(t,x),\\
&\mathrm{div \,} u = 0, \\
& u(0,x)= - u_0(x),
\end{split}
}
where $u_0(x)$ is a divergence-free smooth vector field.
We fix a time interval $[0,T]$, and rewrite equations
\eqref{ns-main} 
with respect to the function
\[
\td u(t,x) = -u(T-t,x).
\]
Problem \eqref{ns-main} 
is equivalent to the following:
\aaa{
\label{ns_b}
\begin{split}
& \frac{\pl}{\pl t}\td u(t,x) = - (\td u,\nab)\td u(t,x)
 -\nu \lap \td u(t,x) - \nab \td p(t,x),\\
& \mathrm{div \,} \td u = 0,\\
& \td u(T,x) = u_0(x),
\end{split}
}
where $\td p(t,x) = p(T-t,x)$.

In what follows, system
 \eqref{ns_b} 
 will be referred to as
 the backward Navier--Stokes equations.
 To this system  we associate a certain
 system of forward-backward stochastic differential
 equations on the group of volume-preserving diffeomorphisms of a flat torus.
 For simplicity, we work in two dimensions. However, the generalization
 of most of the results to the case of $n$ dimensions is straightforward.
 The necessary constructions and non-straightforward generalizations
 related to the $n$-dimensional case are considered in the appendix.

 Assuming the existence of a solution of \eqref{ns_b} 
 with the final data in the Sobolev space $H^\s$ for sufficiently large $\s$,
 we construct a solution of the associated system of FBSDEs.
 Conversely, if we assume that a solution of the system of FBSDEs exists,
 then the solution of the Navier--Stokes equations
 can be obtained from the solution of the FBSDEs.
 In fact, the constructed FBSDEs on the group of volume-preserving diffeomorphisms
 can be regarded as an alternative object
 to the Navier--Stokes equations for studying
 the properties of the latter. 

 The connection between forward-backward SDEs and quasi-linear PDEs
 in finite dimensions has been studied by many authors, for example in
 \cite{Delarue}, \cite{ma}, and \cite{pardoux99}.

 Our construction uses the approach originating in the work of Arnold \cite{arnold1}
 which states that the motion of a perfect fluid can be described in terms of geodesics
 on the group of volume-preserving diffeomorphisms of a compact manifold.
 The necessary differential-geometric structures were developed in later work
 by Ebin and Marsden \cite{ebin_marsden}. We note here that
 \cite{arnold1} and \cite{ebin_marsden} deal only with differential geometry
 on the group of maps without involving probability.

 The associated system of FBSDEs is solved using
 the existence of a solution to \eqref{ns_b}, 
 and by applying results from the works
 of Gliklikh (\cite{Gliklikh3}, \cite{Gliklikh4}, \cite{Gliklikh1}, \cite{Gliklikh2}).
 The latter works use, in turn, the approach
 to stochastic differential equations on Banach manifolds
 developed by Dalecky and Belopolskaya \cite{Belopolskaya}, and
 started by McKean \cite{mckean}. Conversely,  a solution of \eqref{ns_b}
 is obtained using the existence of a solution to the associated FBSDEs as well as
 some ideas and constructions from \cite{Delarue}. However,
 unlike \cite{Delarue}, we work in an infinite-dimensional setting.

 Representations of the Navier--Stokes velocity field as a drift of a diffusion
 process were initiated in \cite{Y} and \cite{N-Y-Z}.
 A different system of stochastic equations (but not a system of two SDEs) 
 associated to the Navier--Stokes system was introduced and studied in \cite{Belopolskaya2}.
 This system also includes an SDE on the group of volume-preserving
 diffeomorphisms, but is not a system of forward-backward SDEs. 
 Also, we mention
 here the works \cite{Belopolskaya_albeverio1} and \cite{Belopolskaya_albeverio2}
 discussing probabilistic representations of solutions to the Navier--Stokes equations,
 and the work \cite{CC} establishing a
 stochastic variational principle for the Navier--Stokes equations.
 Different probabilistic representations of the solution to the Navier--Stokes
 equations were studied for example in \cite{YLeJan} and \cite{Constantin}.
 We note that the list of literature on probabilistic approaches to
 the Navier--Stokes equations as well as connections between finite-dimensional
 FBSDEs and PDEs cited in this paper is by no means complete.

 The method of applying infinite-dimensional forward-backward SDEs
 in connection to the Navier--Stokes
 equations is employed, to the authors' knowledge, for the first time.

 \section{Geometry of the diffeomorphism group 
 of the 2D torus}
 \label{Geometry2D}
 Let $\Tor = S^1 \x S^1$ be the two-dimensional torus, and
 let $H^\s(\Tor)$, $\s > 2$, be the space of $H^\s$-Sobolev maps $\Tor \to\Tor$.
 By $G^{\s}$ we denote the subset of $H^\s(\Tor)$
 whose elements are $\C^1$-diffeomorphisms.
 Let $G^{\,\s}_{\V}$ be the subgroup of $G^{\s}$
 consisting of diffeomorphisms preserving the volume
 measure on $\Tor$.
 \begin{lem}
  \label{Glik_Sob_map}
 Let $g$ be an $H^\s$-map and a local diffeomorphism of a finite-dimensional compact
 manifold $M$, $F$ be an $H^\s$-section of the tangent bundle $TM$.
 Then, $F\circ g$ is an  $H^\s$-map.
 \end{lem}
 \begin{proof}
 See \cite{Gliklikh1} (p. 139) or \cite{ebin_marsden} (p. 108).
 \end{proof}
 Let $R_g$ denote the right translation on $G^\s$, i.e. $R_g(\eta) = \eta\circ g$.
 \begin{lem}
 \label{lem_t1}
 The map $R_g$ is $C^\infty$-smooth for every $g\in G^\s$. 
 Furthermore, for every
 $\eta\in G^\s$, the tangent map $TR_g$ restricted to the tangent
 space $T_\eta G^\s$ is defined by the formula:
 \aa{
  TR_g: \; T_\eta G^\s \to T_{\eta\circ g} G^\s, \;  X \mto  X\circ g.
  }
 \end{lem}
 \begin{proof}
 The proof easily follows from the $\al$-lemma (see \cite{ebin_marsden},
 \cite{Gliklikh1}, \cite{Gliklikh2}).
 \end{proof}
 \begin{lem}
 \label{group_lem1}
 The groups $G^\al$ and $\Gr$ are infinite-dimensional Hilbert manifolds.
 The group  $\Gr$ is a subgroup and a smooth submanifold of $G^\s$.
 \end{lem}
 \begin{lem}
 \label{group_lem2}
 The tangent space $T_eG^\s$ is formed by all $H^\s$-vector fields on 
 $\Tor$.
 The tangent space $T_e\Gr$ is formed by all divergence-free  
 $H^\s$-vector fields on $\Tor$.
 \end{lem}
The proof of Lemmas \ref{group_lem1} and \ref{group_lem2} 
 can be found for example in \cite{ebin_marsden}, \cite{Gliklikh1},
 \cite{Gliklikh2}.
  \begin{lem}
 \label{right-invariant}
 Let $X\in T_eG^\s$ be an $H^\s$-vector field on $\Tor$.
 Then the vector field $\hat X$ on $G^\al$ defined by 
 $\hat X(g) = X\circ g$
 is right-invariant.  Furthermore,
 $\hat X$ is $\C^k$-smooth if and only if $X\in H^{\s+k}$.
 \end{lem}
  \begin{proof}
  The first statement follows from Lemma \ref{lem_t1}. The proof of the  second
  statement can be found in \cite{ebin_marsden}.
  \end{proof}
 The vector field $\hat X$ on $G^\al$ defined in Lemma \ref{right-invariant}
 will be referred to below as the \textit{right-invariant} vector field generated 
 by $X\in T_eG^\s$.

  Let $g\in G^{\, \s}$, $X,Y \in T_e  G^{\s}$.
  Consider the
  weak $(\fdot, \fdot)_0$ and the strong $(\fdot,\fdot)_\s$
  Riemannian metrics on $G^\s$ (see   \cite{Gliklikh2}):
  \aaa{
  \label{weak}
  (\hat X(g),\hat Y(g))_0 = & \int_{\Tor} (X \circ g(\te), Y \circ g(\te)) d\te, \\
  (\hat X(g),\hat Y(g))_\s = & \int_{\Tor} (X \circ g(\te), Y \circ g(\te)) d\te \notag\\
  \label{strong}
   + &\int_{\Tor} ( (d+\dl)^\s X \circ g(\te), (d+\dl)^\s Y \circ g(\te)) d\te
  }
  where  $d$ is the differential,
  $\dl$ is the codifferential, $\hat X$ and $\hat Y$ are the right-invariant vector
  fields on $G^\s$ generated by the $H^\s$-vector fields
  $X$ and $Y$.
  Metric \eqref{weak} gives rise to the $L_2$-topology
  on the tangent spaces of $G^\s$, and metric
  \eqref{strong} gives rise to the $H^\s$-topology on
  the tangent spaces of
  $G^{\s}$ (see \cite{Gliklikh2}).
  If $g\in \Gr$, then scalar products
  \eqref{weak} and \eqref{strong} 
  do not depend on $g$. Moreover, 
  for the strong metric on $\Gr$, we have the following formula:
  \aa{
 (\hat X(g),\hat Y(g))_\s = \int_{\Tor} ( X \circ g(\te), (1+\lap)^\s Y \circ g(\te)) d\te
  }
 where $\lap = (d\dl + \dl d)$ is the Laplace-de Rham operator (see \cite{shkoller}).

 Let us introduce the notation:
 \aa{
  &\Znu_2^+ = \{(k_1,k_2)\in\Znu^2: k_1 > 0 \; \text{or} \; k_1=0, k_2>0\}; \\
  &k=(k_1,k_2)\in \Znu_2^+, \;  \bar k = (k_2, -k_1), \;  |k|=\sqrt{k_1^2 + k^2_2}, \;
  k\cdot\te = k_1\te_1+k_2\te_2, \\
  &\te = (\te_1,\te_2) \in \Tor, \; \nab = \Bigl(\frac{\pl}{\pl\te_1}, \frac{\pl}{\pl\te_2}\Bigr),\;
  (\bar k, \nab) = k_2\frac{\pl}{\pl\te_1} - k_1 \frac{\pl}{\pl\te_2}\, ,
 }
 and the vectors
 \aa{
 \begin{split}
 &\bar A_k(\te) = \frac1{|k|^{\s+1}}
 \cos(k\cdot\theta) \left(\begin{matrix} k_2 \\ - k_1 \end{matrix} \right),
 \quad
 \bar B_k(\te) =
 \frac1{|k|^{\s+1}}
 \sin(k\cdot\theta) \left(\begin{matrix} k_2 \\ - k_1 \end{matrix} \right), \\
 &\bar A_0 = \left(\begin{matrix} 1 \\  0 \end{matrix} \right),
 \quad
 \bar B_0 = \left(\begin{matrix} 0 \\  1 \end{matrix} \right).
 \end{split}
 }
  Let $\{A_k(g), B_k(g)\}_{k\in\Znu_2^+\cup \{0\}}$ be the right-invariant
  vector fields on $G^\al$
  generated by $\{\bar A_k, \bar B_k\}_{k\in\Znu_2^+ \cup \{0\}}$,
  i.e.
  \aa{
  \begin{split}
   & A_k(g)= \bar A_k \circ g, \quad B_k(g) = \bar B_k \circ g, \quad
   g\in G^\s, \\
   & A_0 = \bar A_0, \quad B_0 = \bar B_0.
   \end{split}
  }
   By $\om$-lemma
 (see \cite{Gliklikh1}), $A_k$ and $B_k$ are $\C^\infty$-smooth
 vector fields on $G^\al$.
 \begin{lem}
 \label{Gsv-hilbert}
 The vectors $A_k(g)$, $B_k(g)$, $k\in \Znu_2^+\cup \{0\}$,
 $g\in\Gr$,
 form an orthogonal basis of
 the tangent space $T_g\Gr$
 with respect to both the weak and the strong inner products
 in $T_g\Gr$.
 In particular, the vectors $\bar A_k$, $\bar B_k$,
 $k\in \Znu_2^+\cup \{0\}$,
 form an orthogonal basis of the tangent space $\Alg$.
 Moreover, the weak and the strong norms of the basis vectors are bounded
 by the same constant.
 \end{lem}
 \begin{proof}
 It suffices to prove the lemma for the strong norm.
  Let us compute $\lap^\s \bar A_k$.
  Note that the vectors
  $\frac{k}{|k|}$ and $\frac{\bar k}{|k|}$
  form an orthonormal basis of $\Rnu^2$.
  Let us observe that  by the identity $(\bar k,\nab)\cos(k\cdot\te) = 0$, $\dl \bar A_k = 0$.
  Hence $d\dl \, \bar A_k = 0$ which implies $\lap \, \bar A_k = \dl d \, \bar A_k$.
  We obtain:
  \aa{
  &\bar A_k = \frac1{|k|^\s} \cos(k\cdot\te) \, \frac{\bar k}{|k|}, \\
  & d \,\bar A_k = -\frac1{|k|^{\s-1}} \sin(k\cdot\te)\, \frac{k}{|k|}
  \wedge \frac{\bar k}{|k|}, \\
  & \lap \bar A_k =
  \dl d \, \bar A_k = \frac1{|k|^{\s-2}} \cos(k\cdot\te) \, \frac{\bar k}{|k|}
   = |k|^2 \bar A_k,\\
  & \lap^\s \, \bar A_k = |k|^\s  \cos(k\cdot\te)  \, \frac{\bar k}{|k|}
  =  |k|^{2\s} \bar A_k.
  }
  This and the volume-preserving property of $g\in \Gr$ imply that
  \aa{
  &(B_m(g), A_k(g))_\s = (\bar B_m, \bar A_k)_\s
  = (1+ |k|^{2\s})(\bar B_m, \bar A_k)_{L_2} = 0, \\
  &\|A_k(g)\|^2_\s = \|\bar A_k\|^2_\s = \bigl(1 + |k|^{2\s}\bigr)
   \|\bar A_k\|^2_{L_2} \hspace{-1mm}
   = 2\pi^2  \left(|k|^{-2\s} + 1 \right)
  }
  where $\|\fdot\|_\s$ is the norm corresponding
  to the scalar product $(\fdot,\fdot)_\s$.
  Thus, $2\pi^2 \lt \|A_k(g)\|^2_\s \lt 4\pi^2$.
  Clearly, for the $\|B_k(g)\|^2_\s$ we obtain the same.
\end{proof}
It has been shown, for example, in \cite{ebin_marsden} and \cite{Gliklikh2}
 that the weak Riemannian metric has the Levi-Civita
 connection, geodesics, the exponential map, and the spray.
 Let $\bar \nab$ and $\td \nab$ denote the covariant derivatives of
 the Levi-Civita connection of the weak Riemannian
 metric \eqref{weak} on $G^\s$ and $\Gr$, respectively.
 In \cite{ebin_marsden} (see also \cite{Gliklikh2}, \cite{Gliklikh1}), 
 it has been shown that
 \aa{
  \td \nab = \Proj\circ \bar \nab
 }
where $\Proj : TG^\s \to T\Gr$ is defined in the following way:
 on each tangent space $T_gG^\s$, $\Proj = P_g$ where
 $P_g = TR_g \circ  P_e \circ  TR_{g^{-1}}$, 
 $TR_g$ and $TR_{g^{-1}}$ are tangent maps, 
 and $P_e: T_e G^\s \to \Alg$ is the projector defined by
 the Hodge decomposition.
 \begin{lem}
 \label{lem5'}
 Let $\hat U$ be the right-invariant vector
 field on $G^\s$ generated by an $H^{\s+1}$-vector field $U$ on $\Tor$,
 and let $\hat V$ be the right-invariant vector field on $G^\s$
 generated by an $H^\s$-vector field $V$ on $\Tor$.
 Then $\bar \nab_{\hat V}  \hat U$  
 is the right-invariant vector field
 on $G^\s$ generated by the $H^\s$-vector field $\nab_{V} U$ 
 on $\Tor$.
 \end{lem}
 \begin{lem}
 \label{lem5''}
 Let $\hat U$ be the right-invariant vector
 field on $\Gr$ generated by a divergence-free 
 $H^{\s+1}$-vector field $U$ on $\Tor$, and let 
 $\hat V$ be the right-invariant vector field on $G^\s$
 generated by a divergence-free  $H^\s$-vector field $V$ on $\Tor$.
 Then $\td \nab_{\hat V} \hat U$  is the right-invariant  vector field
 on $\Gr$ generated by the divergence-free  $H^\s$-vector
 field $P_e \nab_V U$ on $\Tor$. 
 \end{lem}
 The proofs of Lemmas \ref{lem5'} and \ref{lem5''} follow from
 the right-invariance of covariant derivatives
 on $G^\s$ and $\Gr$ (see \cite{Gliklikh2}).
 \begin{rem}\label{remark1}{\rm
 The basis $\{\bar A_k, \bar B_k\}_{k\in \Znu_2^+ \cup \{0\}}$ of $\Alg$ can be
 extended to a basis of the entire tangent space $T_eG^\s$. Indeed, let us
 introduce the vectors:
  \aa{
   \bar{\mc A_k}(\te) = \frac1{|k|^{\s+1}}\cos(k\cdot\te)
   \left(\begin{matrix} k_1 \\ k_2 \end{matrix}\right), \;
    \bar{\mc B_k}(\te) = \frac1{|k|^{\s+1}}\sin(k\cdot\te)
   \left(\begin{matrix} k_1 \\ k_2 \end{matrix}\right)
   , \; k\in \Znu^+_2.
  }
  The system $\bar A_k$, $\bar B_k$, $k\in \Znu^+_2 \cup \{0\}$,
  $\bar{\mc A_k}$, $\bar{\mc B_k}$, $k\in \Znu^+_2$, form an orthogonal basis of $T_eG^\s$.
  Further let
  $\mc A_k$ and $\mc B_k$ denote the right-invariant vector fields
  on $G^\s$ generated by $\bar{\mc A_k}$ and $\bar{\mc B_k}$.}
 \end{rem}

\section{The FBSDEs on the group of
  diffeomorphisms of the 2D torus}

 Let $h: \Tor \to \Rnu^2$
 be a divergence-free $H^{\s+1}$-vector field on $\Tor$,
 and let $\hat h$ be
 the right-invariant vector field on $G^\al$ generated by $h$.
 Further let the function
 $V(s,\fdot)$ be such that there exists a function $p: [t,T] \to H^{\s+1}(\Tor,\Rnu)$
 satisfying $V(s, \fdot) = \nab p (s,\fdot)$ for  all $s\in [t,T]$.
 For each $s\in [t,T]$,  $\hat V(s,\fdot)$  denotes 
 the right-invariant vector field on $G^\s$ generated
 by $V(s,\fdot)\in H^{\s}(\Tor,\Rnu^2)$.

 Let $E$ be a Euclidean space spanned on an orthonormal,
 relative to the scalar product in $E$, system of vectors
 $\{e_k^A, e_k^B, e_0^A, e_0^B\}_{k\in \Znu_2^+, |k|\lt N}$.
 Consider the map
 \aa{
  \sg(g) = \sum_{\substack{k\in \Znu_2^+\cup \{0\}, \\ |k|\lt N}}
     A_k(g) \otimes e_k^A + B_k(g) \otimes e_k^B, \; g\in G^\s,
 }
 i.e. $\sg(g)$ is a linear operator $E\to T_g G^\s$ for each $g\in G^\s$.

 Let $(\Om,\mc F, \PP)$ be a probability space, and
 $W_s$, $s\in [t,T]$, be an $E$-valued Brownian motion:
 \aa{
 W_s = \sum_{\substack{k\in \Znu_2^+\cup \{0\}, \\ |k|\lt N}}
 (\beta^A_k(s)e^A_k  + \beta^B_k(s)e^B_k)
 }
 where $\{\beta^A_k,\beta^B_k\}_{k\in \Znu_2^+\cup\{0\}, |k|\lt N}$ is
 a sequence of independent Brownian motions.
  We consider the following system of forward and backward SDEs:
 \eee{
 \label{fbsde-torus}
 \begin{cases}
  dZ_s^{t,e} = Y_s^{t,e} ds + \epsilon\,\sg(Z_s^{t,e}) dW_s, \\
  dY_s^{t,e} = -\hat V(s,Z_s^{t,e}) ds + X_s^{t,e}  dW_s, \\
  Z^{t,e}_t = e; \; Y_T^{t,e} = \hat h(Z_T^{t,e}).
 \end{cases}
 }
 The forward SDE  of \eqref{fbsde-torus}
 is an SDE on $\Gr$ where $\Gr$ is considered as a Hilbert manifold.
 Stochastic differentials
 and stochastic differential equations on Hilbert manifolds
 are understood in the sense of
 Dalecky and Belopolskaya's approach (see \cite{Belopolskaya}).
 More precisely, we use the results
 from \cite{Gliklikh1}
 which interprets the latter approach
 for the  particular case
 of SDEs on Hilbert manifolds.
 The stochastic integral in the forward SDE
  can be explicitly written as follows:
  \aaa{
  \label{sto_int}
  \int_t^s \sg(Z^{t,e}_r) dW_{r} = \sum_{k\in \Znu_2^+ \cup \{0\}, |k|\lt N}
  \int_t^s A_k(Z^{t,e}_r)d\beta^A_k(r) + B_k(Z^{t,e}_r)d\beta^B_k(r).
  }
Let us consider the backward SDE:
 \aaa{
 \label{bwd_sde}
 Y_s^{t,e} = \hat h(Z_T^{t,e}) + \int_s^T \hat V(r,Z^{t,e}_r) dr
 - \int_s^T X_r^{t,e} dW_r.
 }
 Note that the processes 
 $\hat V(s, Z^{t,e}_s) = V(s,\fdot) \circ Z^{t,e}_s$
 and $\hat h(Z^{t,e}_T) = h\circ Z^{t,e}_T$ are $H^\s$-maps by Lemma \ref{Glik_Sob_map}.
 Therefore, it makes sense to understand SDE \eqref{bwd_sde}
 as an SDE in the Hilbert space $H^\s(\Tor,\Rnu^2)$.
 Let $\mc F_s = \sg(W_r, r \in [0,s])$.
 We would like to find an $\mc F_s$-adapted
 triple of
 stochastic processes $(Z_s^{t,e}, Y_s^{t,e}, X_s^{t,e})$
 solving FBSDEs \eqref{fbsde-torus} in the following sense:
 at each time $s$, the process $(Z_s^{t,e}, Y_s^{t,e})$ takes values in an
 $H^\s$-section of the tangent bundle $T\Gr$.
 Namely, for each $s\in [t,T]$ and $\om \in \Om$,
 $Z_s^{t,e}\in \Gr$, $Y_s^{t,e}\in T_{Z_s^{t,e}}\Gr$. 
 Therefore, the forward SDE is well-posed on both $G^\al$ and $\Gr$, and
 can be written in the Dalecky--Belopolskaya form:
 \aa{
 dZ_s^{t,e} = \exp_{Z_s^{t,e}}&\{Y_s^{t,e} ds + \epsilon\,\sg(Z_s^{t,e}) dW_s\}\\
 &\text{or} \notag\\
 dZ_s^{t,e} =  \td{\exp}_{Z_s^{t,e}}&\{Y_s^{t,e} ds + \epsilon\,\sg(Z_s^{t,e}) dW_s\}
 }
 where $\exp$ and $\td{\exp}$ are the exponential maps of the Levi-Civita
 connection of the weak Riemannian metrics
 \eqref{weak} on $G^\s$ and resp. $\Gr$.
 Below, we will show that using either of these representations leads to the same solution
 of FBSDEs \eqref{fbsde-torus}. 
 
%

 Finally, the process $X_s^{t,e}$
 takes values in the space of linear operators $\mc L(E, H^\s(\Tor,\Rnu^2))$, i.e.
 \aaa{
 \label{process-x}
  X_s^{t,e} = \sum_{k\in\Znu_2^+\cup\{0\}, |k|\lt  N} X^{kA}_s\ox e^A_k + X^{kB}_s\ox e^B_k
 }
 where the processes $X^{kA}_s$ and  $X^{kB}_s$ take values in $H^\s(\Tor,\Rnu^2)$. 
\begin{rem}
 {\rm 
 The results obtained below also work in the situation when the
 Brownian motion $W_s$ is infinite dimensional (as in \cite{ABC_PM}). 
 Namely, when 
 $W_s = \sum_{k\in \Znu_2^+\cup\{0\}} a_k \beta_k^A \ox e^A_k + 
  b_k \beta_k^B \ox e^B_k$ where $a_k$, $b_k$, $k\in \Znu_2^+\cup\{0\}$,
  are real numbers satisfying $\sum_{k\in \Znu_2^+\cup\{0\}}|a_k|^2 + |b_k|^2 <\infty$.
  However, this requires an additional analysis on the solvability of the forward SDE
  based on the approach of Dalecky and Belopolskaya \cite{Belopolskaya}
  since the results of Gliklikh (\cite{Gliklikh3}, \cite{Gliklikh1}, \cite{Gliklikh2})
  applied below 
  are obtained for the case of a finite-dimensional Brownian motion.
 }
 \end{rem}

 \section{Constructing a solution of the FBSDEs}

 \subsection{The forward SDE}

  Let us consider the backward Navier--Stokes equations in $\Rnu^2$:
  \aaa{
  \label{ns_consider}
  \begin{split}
  & y(s,\te) = h(\te) + \int_s^T \bigl[\nab p(r,\te) +
  \bigl(y(r,\te),\nab\bigr)y(r,\te)
  + \nu \lap  y(r,\te)\bigr]dr,\\
  &\div y(s,\te) = 0
  \end{split}
   }
  where $s\in [t,T]$, $\te\in \Tor$, 
  $\lap$ and $\nab$ are the Laplacian and the gradient.
  \begin{asm}
  \label{asn1}
  Let us assume that on the interval $[t,T]$ 
  there exists a solution $\bigl(y(s,\fdot),  p(s, \fdot)\bigr)$
  to \eqref{ns_consider}  such that
  the functions $p: [t,T] \to H^{\s+1}(\Tor,\Rnu)$
  and $y: [t,T] \to H^{\s+1}(\Tor,\Rnu^2)$ are continuous. 
  \end{asm}
  Clearly, $y(s,\fdot) \in \Alg$.
  Let $\{Y^{t;kA}_s, Y^{t;kB}_s\}_{k\in\Znu_2^+ \cup \{0\}}$ be the coordinates of $y(s,\fdot)$
  with respect to the basis $\{\bar A_k, \bar B_k\}_{k\in\Znu_2^+ \cup \{0\}}$, i.e.
  \[
  y(s,\te) = \sum_{k\in\Znu_2^+ \cup \{0\}} Y^{t;kA}_s \bar A_k(\te) +  Y^{t;kB}_s \bar B_k(\te).
  \]
  Let $\hat Y_s(\fdot)$ denote the right-invariant vector field on $G^\s$
  generated by the solution $y(s,\fdot)$, i.e. $\hat Y_s(g) = y(s,\fdot) \circ g$.
  On each tangent space $T_gG^\s$, the vector $\hat Y_s(g)$
  can be represented by a series converging in the $H^\s$-topology:
  \aaa{
  \label{Y-repres}
  \hat Y_s(g) =  \sum_{k\in\Znu_2^+\cup \{0\}} Y^{t;kA}_s A_k(g) +  Y^{t;kB}_s B_k(g).
  }
  In this paragraph we will study the SDE:
  \aaa{
  \label{forward_sde}
  dZ^{t,e}_s = \hat Y_{s}(Z^{t,e}_s) ds
  + \epsilon\,\sg(Z^{t,e}_s)dW_{s}.
  }
  Later, in Theorem \ref{thm2},  we will show that the solution $Z^{t,e}_s$ to \eqref{forward_sde}
  and the process $Y^{t,e}_s = \hat Y_s(Z^{t,e}_s)$ are the first two processes in the triple
  $(Z^{t,e}_s, Y^{t,e}_s, X^{t,e}_s)$ that solves FBSDEs \eqref{fbsde-torus}.
  \begin{thm}
  \label{forward_esistence1}
  There exists
  a unique strong solution $Z^{t,e}_s$, $s\in [t,T]$,
  to \eqref{forward_sde} on $\Gr$,
  with the initial condition $Z^{t,e}_t = e$.
  \end{thm}
  \begin{proof}
  Below, we verify the assumptions of Theorem 13.5 of \cite{Gliklikh2}.
  The latter theorem will imply the existence and uniqueness of the strong
  solution to \eqref{forward_sde}. Note that,
  if sum \eqref{sto_int} representing the stochastic integral
  $\int_t^s\sg(Z^{t,e}_s)\, dW_s$ contains only the terms $A_0 (\beta^A_0(s)-\beta^A_0(t))$  and
  $B_0 (\beta^B_0(s)- \beta^B_0(t))$,
  i.e., informally speaking, if
  the Brownian motion runs only along the constant vectors $A_0$ and $B_0$,
  then the statement of the theorem follows from Theorem 28.3 of \cite{Gliklikh2}.
  If sum \eqref{sto_int} contains also terms
   with $A_k$ and $B_k$, $k\in\Znu_2^+$,
  or, informally, when the Brownian motion runs also along 
  non-constant vectors
  $A_k$ and $B_k$, $k\in\Znu_2^+$,
  then the assumptions of Theorem 13.5 of \cite{Gliklikh2}
  require the boundedness of $A_k$ and $B_k$ with respect
  to the strong norm. The latter fact holds by Lemma
  \ref{Gsv-hilbert}.

  Hence, all the assumptions of Theorem 13.5 of \cite{Gliklikh2}
  are satisfied. Indeed,
  the proof of Theorem 28.3 of \cite{Gliklikh2} shows
  that the Levi-Civita connection of the weak Riemannian metric
  \eqref{weak} on $\Gr$ is compatible (see Definition 13.7 of \cite{Gliklikh2})
  with the strong Riemannian metric \eqref{strong}.
   The function $\sg(g) = \sum_{k\in \Znu_2^+\cup \{0\}, |k|\lt N} A_k(g) \ox e_k^A + B_k(g) \ox e_k^B$
  is $\C^\infty$-smooth since $A_k$ and $B_k$ are $\C^\infty$-smooth. 
  Moreover, by Lemma \ref{Gsv-hilbert}, $\sg(g)$ is bounded on $\Gr$.
  Next, since $y: [t,T] \to H^{\s+1}(\Tor,\Rnu^2)$ is continuous, then it is also bounded
 with respect to (at least) the $H^\s$-norm. 
 Hence, the generated right-invariant
  vector field $\hat Y_s(g)$ is bounded in $s$ 
  with respect to the strong metric \eqref{strong},
  and it is at least $\C^1$-smooth in $g$.
 The boundedness of $\hat Y_s$ in $g$ follows from the volume-preserving
  property of $g$. 
  \end{proof}
  \begin{thm}
  \label{existence2}
  There exists a unique strong solution $Z^{t,e}_s$, $s\in [t,T]$,
  to \eqref{forward_sde} on $G^\s$,
  with the initial condition $Z^{t,e}_t = e$. This solution coincides
  with the solution to SDE \eqref{forward_sde} on $\Gr$.
  \end{thm}
  \begin{proof}
  Consider the identical embedding $\imath: \Gr \to G^\s$.
  By results of \cite{Belopolskaya} (Proposition 1.3, p. 146; see also \cite{Gliklikh2}, p. 64), 
  the stochastic process $\imath(Z^{t,e}_s)= Z^{t,e}_s$, $s\in [t,T]$, is a solution
  to SDE \eqref{forward_sde} on $G^\s$, i.e. 
  with respect to the exponential map $\exp$. 
  This easily follows from the fact that $T\imath: T\Gr \to TG^\s$, 
  where $T$ is the tangent map, is the identical imdedding,
  and that $\imath\bigl(\exp(X)\bigr) = \td{\exp}(T\imath \circ X)$. 
  The solution $Z^{t,e}_s$ to \eqref{forward_sde} on $G^\s$  is unique.
  This follows from the uniqueness theorem
  for SDE \eqref{forward_sde} considered on the manifold
  $G^\al$ equipped with the weak Riemannian metric. 
  Indeed, $\sg(g)$ and $\hat Y_s(g)$ are bounded
  with respect to the weak metric \eqref{weak} since
  the functions $\bar A_k$, $\bar B_k$, $k\in \Znu_2^+\cup \{0\}$,
  are bounded on $\Tor$,
  and $y(\fdot,\fdot)$ is bounded on $[t,T]\x \Tor$.
  Moreover $\sg(g)$ is $\C^\infty$-smooth
  and $\hat Y_s$ is at least $\C^1$-smooth on $G^\al$.
  \end{proof}
  One can also consider \eqref{forward_sde} as an
  SDE with values in the Hilbert space $H^\s(\Tor,\Rnu^2)$. 
  \begin{thm}
  \label{solution2}
  There exists a unique strong solution 
  $Z^{t,e}_s$ to the  $H^\s(\Tor,\Rnu^2)$-valued SDE
  \eqref{forward_sde} on $[t,T]$,
  with the initial condition $Z^{t,e}_t = e$ where $e$ is the identity
  of $\Gr$.
  This solution coincides
  with the solution to SDE \eqref{forward_sde} on $\Gr$ or $G^\s$.
  \end{thm}
  \begin{proof}
  By Theorem \ref{forward_esistence1}, SDE \eqref{forward_sde}
  on $\Gr$ has a unique strong solution $Z^{t,e}_s$ on $[t,T]$.
  Let us prove that the solution $Z^{t,e}_s$ to \eqref{forward_sde}
  solves this SDE considered as an 
  SDE in $H^\s(\Tor,\Rnu^2)$.
  Consider the identical imbedding
  $\imath_{\sss V}:\, \Gr \to H^\s(\Tor,\Rnu^2)$, $g \mto g$. 
  Applying It\^o's formula to $\imath_{\sss V}$,
  and taking into account that 
  $A_k(g) \imath_{\sss V}(g) = \nab_{\bar A_k}\te \circ g = A_k(g)$
 and that $A_k(g)A_k(g) \imath_{\sss V}(g) = A_k(g)A_k(g) = 0$,
 we obtain that the solution $Z^{t,e}_s$ to \eqref{forward_sde}
  on $\Gr$ solves the $H^\s(\Tor,\Rnu^2)$-valued
  SDE \eqref{forward_sde}. 
  Note that by the uniqueness theorem for SDEs in Hilbert spaces,
  SDE \eqref{forward_sde} can have only one solution in $L_2(\Tor,\Rnu^2)$.
  This proves the uniqueness of its solution in $H^\s(\Tor,\Rnu^2)$ as well.
  Thus the solutions to
  \eqref{forward_sde} on $G^\s$, $\Gr$, and in $H^\s(\Tor,\Rnu^2)$ coincide.
  \end{proof}

Let us find the representations of SDE \eqref{forward_sde} in normal coordinates
on $G^\s$ and $\Gr$.
  First, we prove the following lemma.
  \begin{lem}
  \label{lem5}
  The following equality holds:
  \aa{
  \int_t^{s}\sg(Z^{t,e}_r)\circ dW_r = \int_t^{s}\sg(Z^{t,e}_r)dW_r, 
  }
  i.e. instead of the It\^{o} stochastic integral in \eqref{forward_sde}
  we can write
  the Stratonovich stochastic integral $\int_t^s \sg(Z^{t,e}_r)\circ dW_r$.
  \end{lem}
 \begin{proof}
 We have:
 \aa{
 \sg(Z^{t,e}_r)\circ dW_{r} = \sg(Z^{t,e}_r) dW_{r}
 + \hspace{-3mm}
 \sum_{k\in \Znu_2^+ \cup \{0\}, |k|\lt N} \hspace{-3mm}
  dA_k(Z^{t,e}_r)d\beta^A_k(r) + dB_k(Z^{t,e}_r)d\beta^B_k(r).
 }
 Hence, we have to prove that $dA_k(Z^{t,e}_r)d\beta^A_k(r) = 0$
 and $dB_k(Z^{t,e}_r)d\beta^B_k(r)= 0$.
 For simplicity of notation we use the
 notation $A_\nu$ for both of the vector fields
 $A_k$ and $B_k$ and the notation $\bar A_\nu$ for $\bar A_k$ and $\bar B_k$,
 $k\in\Znu^+_2\cup \{0\}$.
 Also, we use the notation
 $\beta_\nu(s)$ for the Brownian motions
 $\{\beta^A_k(s),\beta^B_k(s)\}_{k\in\Znu^+_2\cup\{0\}, |k|\lt N}$.
 We obtain:
 \aa{
 d(\bar A_\nu \circ Z^{t,e}_s) =
 \sum_\gm A_\gm(Z^{t,e}_s)\bigl(\bar A_\nu\circ Z^{t,e}_s \bigr)\circ d\beta_\gm(s)
 + Y^{t,e}_s\bigl(\bar A_\nu\circ Z^{t,e}_s \bigr) dt.
 }
 This implies
 \aa{
 d(\bar A_\nu \circ Z^{t,e}_s)\cdot d\beta_\nu =
 A_\nu(Z^{t,e}_s)\bigl(\bar A_\nu\circ Z^{t,e}_s \bigr)ds = 0
 }
 which holds
 by the identity $(\bar k,\nab)\cos(k\cdot\te) = (\bar k,\nab)\sin(k\cdot\te)= 0$
 or by differentiating of constant vector fields.
 \end{proof}
Let $\bar Z^{t}_s = \{Z_s^{t;kA}, Z_s^{t;kB}\}_{k\in \Znu_2^+ \cup \{0\}}$
 be the vector of local coordinates of the solution $Z^{t,e}_s$
 to \eqref{forward_sde} on $\Gr$, i.e. the vector of normal
 coordinates provided by the exponential
 map $\td{\exp} :\Alg \to \Gr$. Let $U_e$ be the
 canonical chart of the map $\td{\exp}$.
  \begin{thm}[SDE \eqref{forward_sde} in local coordinates]
 \label{Z-coord}
  Let
  \aaa{
  \label{tau1stopping}
  \tau = \inf\{s\in [t,T]:  Z_s^{t,e} \notin U_e\}.
  }
  On the interval $[t,\tau]$, SDE \eqref{forward_sde}
  has the following representation in local coordinates:
  \aaa{
  \label{FSDE-simple}
  \begin{split}
  &Z^{t, kA}_{s\we\tau} = \int_{t}^{{s}\we\tau}
  Y^{t; kA}_{r} dr + \dl_k\epsilon \,(\beta_k^A(s\we\tau) - \beta_k^A(t)),\\
  &Z^{t, kB}_{s\we\tau} = \int_{t}^{{s}\we\tau}
  Y^{t; kB}_{r} dr + \dl_k\epsilon \,(\beta_k^B(s\we\tau) - \beta_k^B(t)).
  \end{split}
  }
  where $\dl_k = 1$ if $|k|\lt N$, and $\dl_k = 0$ if $|k| > N$.
 \end{thm}
  \begin{proof}
  Let $\bar g = \{g^{kA},g^{kB}\}_{k\in \Znu_2^+ \cup \{0\}}$
  be local coordinates
  in the neighborhood $U_e$ provided by the map $\td {\exp}$.
  Let $f \in \C^\infty(G^{\,\s}_{\V})$, and let $\td f: T_e\Gr \to \Rnu$
  be such that $\td f = f\circ \td{\exp}$. Since
  $\td{\exp}$ is a $C^\infty$-map (see \cite{ebin_marsden}), 
  then $\td f \in \C^\infty(U_0)$, where $U_0 = \td{\exp}^{-1}U_e$.
  Note that
    $\frac{\pl}{\pl g^{kA}}\td f(\bar g) = A_k(g)f(g)$ and $\frac{\pl}{\pl g^{kB}}\td f(\bar g) = B_k(g)f(g)$.
  By It\^{o}'s formula, we obtain:
  \mm{
  f(Z^{t,e}_{s\we\tau}) - f(e) =
  \td f(\bar Z^{t,0}_{s\we\tau}) - \td f(0)
  \\
  =\int_{t}^{s\we\tau} dr \sum_{k \in\Znu_2^+ \cup\{0\}}
  \frac{\pl \td f}{\pl g^{kA}}(\bar Z^{t}_r)  Y^{t;kA}_r
  + \int_t^{s\we\tau} \hspace{-1mm} \epsilon \hspace{-2mm} \sum_{k \in\Znu_2^+ \cup\{0\}}
  \dl_k\,\frac{\pl \td f}{\pl g^{kA}}(\bar Z^{t}_r) Y^{t;kA}_r
  \circ  d\beta_k^A(r) \\
  +\int_{t}^{s\we\tau} dr \sum_{k \in\Znu_2^+ \cup\{0\}}
  \frac{\pl \td f}{\pl g^{kB}}(\bar Z^{t}_r)  Y^{t;kB}_r
  + \int_t^{s\we\tau} \hspace{-1mm} \epsilon  \hspace{-2mm} \sum_{k \in\Znu_2^+ \cup\{0\}}
  \dl_k\,\frac{\pl \td f}{\pl g^{kB}}(\bar Z^{t}_r)  Y^{t;kB}_r
  \circ d\beta_k^B(r) \\
  =\int_t^{s\we\tau}  dr \,
  \sum_{k\in \Znu_2^+ \cup\{0\}} \bigl(Y^{t;kA}_rA_k(Z^{t,e}_r) f(Z^{t,e}_r)+
  Y^{t;kB}_rB_k(Z^{t,e}_r\bigr)
  f(Z^{t,e}_r)\bigr)\\
  + \int_t^{s\we\tau} \hspace{-1mm} \epsilon \hspace{-2mm} \sum_{k\in \Znu_2^+\cup\{0\}} \dl_k
  \,\bigl(A_k(Z^{t,e}_r) f(Z^{t,e}_r) \circ d\beta^{A}_k(r)+
  B_k(Z^{t,e}_r) f(Z^{t,e}_r) \circ d\beta_k^{B}(r)\bigr).
  }
  Using representations \eqref{Y-repres} and \eqref{sto_int} we obtain:
  \aa{
  f(Z^{t,e}_{s\we\tau}) - f(e) = \int_t^{s\we\tau} \hat Y_r(Z^{t,e}_r) f(Z^{t,e}_r) dr
  + \int_t^{s\we\tau} \epsilon \, \sg(Z^{t,e}_r) f(Z^{t,e}_r)\circ dW_r.
  }
  This shows that the process
  \aa{
   \exp\Bigl\{\sum_{k\in \Znu_2^+\cup \{0\}}Z^{t, kA}_{s\we\tau}\bar A_k
  + Z^{t, kB}_{s\we\tau}\bar B_k \Bigr\}
  }
  solves SDE \eqref{forward_sde} on the interval $[t,\tau]$.
  \end{proof}
 Let
 \aa{
 \check{\bar Z}_s^{t} =  \{\check Z_s^{t;kA}, \check Z_s^{t;kB}, \check Z_s^{t;k\mc A}, \check Z_s^{t;k\mc B},
  \check Z_s^{t;0A}, \check Z_s^{t;0B}\}_{k\in \Znu_2^+}
 }
 be the vector of local coordinates of the solution $Z^{t,e}_s$
 to \eqref{forward_sde} on $G^\s$, i.e. the vector of normal coordinates provided by the
 exponential map $\exp: T_e G^\s \to G^\s$. Further let
 $\check U_e$ be the canonical chart of the map $\exp$.
  \begin{thm}
  Let
  \aa{
  \check\tau = \inf\{s\in [t,T]:  Z_s^{t,e} \notin \check U_e\}.
  }
  Then, a.s. $\check \tau = \tau$, where the stopping time $\tau$
  is defined by  \eqref{tau1stopping}, and on $[t,\tau]$, 
  $\check Z_s^{t;kA} = Z_s^{t;kA}$,  $\check Z_s^{t;kB} = Z_s^{t;kB}$,
  $k\in \Znu^2_+\cup \{0\}$, $\check Z_s^{t;k\mc A}= \check Z_s^{t;k\mc B} = 0$,
  $k\in \Znu_2^+$, a.s.
 \end{thm}
 \begin{proof}
  Let us introduce 
  additional local coordinates $g^{k\mc A}, g^{k\mc B}$, $k\in \Znu_2^+$,
  and perform the same computation as in the proof of Theorem \ref{Z-coord}.
  We have to take into account that $Y^{k\mc A}_s = Y^{k\mc B}_s = 0$, $k\in \Znu_2^+$,
  and that the components of the Brownian motion are 
  non-zero only along divergence-free
  and constant vector fields. We obtain that 
  the coordinate process $\check{\bar Z}_s^{t}$ verifies 
  SDEs  \eqref{FSDE-simple} and the equations 
  $\check Z_s^{t;k\mc A}= \check Z_s^{t;k\mc B} = 0$, $k\in \Znu_2^+$.
 \end{proof}

 \subsection{The backward SDE and the solution of the FBSDEs}

 We have the following result:
 \begin{thm}
 \label{thm2}
 Let $\hat Y_s$ be the right-invariant vector field generated
 by the solution $y(s,\fdot)$ to the backward Navier--Stokes equations
 \eqref{ns_consider}.
 Further let $Z^{t,e}_s$ be the solution to SDE \eqref{forward_sde}
 on $\Gr$.
 Then there exists an $\epsilon >0$
 such that the triple of stochastic processes
 \aa{
 Z^{t,e}_s, \,  Y^{t,e}_s = \hat Y_s(Z^{t,e}_s), \,
 X^{t,e}_s = \epsilon\, \sg(Z^{t,e}_s)\hat Y_s(Z^{t,e}_s)
 }
 solves FBSDEs \eqref{fbsde-torus} on the interval $[t,T]$.
 \end{thm}
\begin{rem}
\label{remark2}
{\rm
 The expression $\sg(Z^{t,e}_s)\hat Y_s(Z_{s}^{t,e})$
 means the following:
 \aa{
\sg(Z^{t,e}_s)\hat Y_{s}(Z_{s}^{t,e}) 
  = \hspace{-3mm} \sum_{k\in \Znu_2^+\cup\{0\}, |k| \lt N} \hspace{-2mm}
  A_k(Z^{t,e}_s)\hat Y_{s}(Z_{s}^{t,e})\ox e^A_k
 + B_k(Z^{t,e}_s)\hat Y_{s}(Z_{s}^{t,e})\ox e^B_k 
 }
 where $\hat Y_s(\fdot)$ is regarded as a function $\Gr \to H^\s(\Tor,\Rnu^2)$,
 and  $A_k(g)\hat Y_s(g)$ means differentiation of $\hat Y_s :\, \Gr \to H^\s(\Tor,\Rnu^2)$
 along the vector field $A_k$ at the point $g\in \Gr$.
 Let $\gm_\xi$ be the geodesic in $\Gr$ such that $\gm_0 = e$ and $\gm'_0 = \bar A_k$.
 We obtain:
 \mmm{
 \label{diff_along_vfield}
 A_k(g)\hat Y_s(g)(\te) = \frac{d}{d\xi} \,\hat Y_s(\gm_\xi\circ g)(\te)|_{\xi=0}
 = R_g \,\frac{d}{d\xi}\, y(s,\gm_\xi \te)|_{\xi = 0} \\ = R_g \,\nab_{\bar A_k}y(s,\te) 
 = \bar\nab_{A_k}\hat Y_s(g)(\te).
 }
 Thus,
 \aaa{
\label{X}
X^{t,e}_s = 
  \epsilon \hspace{-2mm}\sum_{k\in \Znu_2^+\cup\{0\}, |k| \lt N}
 [\nab_{ \bar A_k} y(s,\fdot) \ox e^A_k + \nab_{ \bar B_k} y(s,\fdot) \ox e^B_k]\circ Z_{s}^{t,e},
 }
 and the stochastic integral in \eqref{bwd_sde} can be represented as
 \mm{
 \int_s^T X^{t,e}_r dW_r \\ = \epsilon \hspace{-2mm} \sum_{k\in \Znu_2^+\cup\{0\}, |k| \lt N}
 \int_s^T \nab_{ \bar A_k} y(r,\fdot) \circ Z_{r}^{t,e} d\beta^A_k(r)
 + \int_s^T \nab_{ \bar B_k} y(r,\fdot) \circ Z_{r}^{t,e} d\beta^B_k(r).
 }
 In particular, if $N=0$,
 \aa{
 \int_s^T X^{t,e}_r dW_r =
 \epsilon \left(
 \int_s^T \hspace{-2mm}\frac{\pl}{\pl \te_1} y(r,\fdot) \circ Z_{r}^{t,e} d\beta^A_0(r)
 + \int_s^T \hspace{-2mm} \frac{\pl}{\pl \te_2} y(r,\fdot) \circ Z_{r}^{t,e} d\beta^B_0(r)
 \right).
 }
}
\end{rem}
A result similar to Lemma \ref{lapl} below  was obtained in \cite{CC}.
 \begin{lem}[The Laplacian of a right-invariant vector field]
 \label{lapl}
 Let $\hat V$ be the right-invariant vector field on
 $G^{\td \al}$ generated by an $H^{\td{\al}+2}$-vector field $V$
 on $\Tor$. Further let  $\epsilon > 0$ be such that
  $\frac{\epsilon^2}2 \Bigl(1+\frac12\sum_{k\in\Znu_2^+, |k|\lt N}\frac1{|k|^{2\s}}\Bigr) = \nu$.
  Then for all $g\in G^{\td\al}$,
  \aaa{
  \label{lap}
   \frac{\epsilon^2}2 \sum_{\substack{k\in\Znu_2^+ \cup\{0\},\\ |k|\lt N}}
   \bigl(\bar \nab_{A_k}\bar \nab_{A_k}
   +\bar \nab_{B_k}\bar \nab_{B_k}\bigr)\,\hat V (g)
   = \nu \, \lap V  \circ g.
  }
 \end{lem}
 Here $\td \al$ is an integer which is not necessary equal to $\s$.
 \begin{proof}
 By the right-invariance of the vector fields
 $\bar \nab_{A_k} \bar \nab_{A_k} \hat V$ and
 $\bar \nab_{B_k} \bar \nab_{B_k} \hat V$ (Lemma \ref{lem5'}),
 it suffices to show \eqref{lap} for $g=e$.
 We observe that
 \aa{
  (\bar k,\nab)\cos(k\cdot \te) = (\bar k,\nab)\sin(k\cdot \te) = 0.
  }
 Then, for $k\in \Znu_2^+$, $\te\in\Tor$,
  \aa{
  \bar \nab_{A_k}\bar \nab_{A_k}\hat V(e)(\te) = \frac1{|k|^{2\s+2}}
   \cos(k\cdot \te)(\bar k,\nab)\bigl[
   \cos(k\cdot \te)(\bar k,\nab)V(\te)\bigr] \\
   = \frac1{|k|^{2\s+2}} \cos(k\cdot \te)^2(\bar k,\nab)^2 V(\te).
  }
  Similarly, $\bar \nab_{B_k}\bar \nab_{B_k}\hat V(e)(\te)= \frac1{|k|^{2\s+2}} \sin(k\cdot \te)^2(\bar k,\nab)^2 V(\te)$. Hence, for each $k\in\Znu_2^+$,
  \aaa{
  \label{ABk}
  (\bar \nab_{A_k}\bar \nab_{A_k} + \bar \nab_{B_k}\bar \nab_{B_k})\,\hat V(e)(\te)
  = \frac1{|k|^{2\s+2}} (\bar k,\nab)^2 V(\te).
  }
  Note that for each $k\in\Znu_2^+$, either
  $\bar k$ or $-\bar k$ is in $\Znu_2^+$, and
  \aa{
  (\bar k,\nab)^2 + (k,\nab)^2 = |k|^2\lap.
  }
  Summation over $k\in\Znu_2^+$, $|k|\lt N$,
  in \eqref{ABk},
  and coupling the terms numbered by $k$ and $\bar k$ (or
  $-\bar k$) gives:
  \aa{
  \sum_{k\in\Znu_2^+,|k|\lt N}
  (\bar \nab_{A_k}\bar \nab_{A_k} + \bar \nab_{B_k}\bar \nab_{B_k})\hat V(e)(\te)
  = \frac12 \sum_{k\in\Znu_2^+, |k|\lt N}\frac1{|k|^{2\s}}\,\lap V(\te).
  }
  Note that $(\bar \nab_{A_0}\bar \nab_{A_0} + \bar \nab_{B_0}\bar \nab_{B_0})\hat V(e)(\te) = \lap V(\te)$.
  Finally, we obtain:
  \aa{
  \sum_{\substack{k\in\Znu_2^+ \cup\{0\},\\ |k|\lt N}}
  (\bar \nab_{A_k}\bar \nab_{A_k}+ \bar \nab_{B_k}\bar \nab_{B_k})\hat V(e)(\te)
  = \Bigl(1 + \frac12\sum_{k\in\Znu_2^+, |k|\lt N}\frac1{|k|^{2\s}}\Bigr) \lap V(\te).
  }
  The lemma is proved.
 \end{proof}
\begin{cor}
\label{cor_laplacian}
Let the function $\ffi:\Tor\to\Rnu^2$ be $\C^2$-smooth.
Further let $A_k(g)[\ffi\circ g]$ and $B_k(g)[\ffi\circ g]$, $k\in\Znu^+_2$, mean the differentiation
of the function $G^{\td \al} \to L_2(\Tor,\Rnu^2)$, $g\mto \ffi\circ g$
along $A_k$ and resp. $B_k$. Then for all $g\in G^{\td\al}$,
\aaa{
\label{lap_smooth_f}
   \frac{\epsilon^2}2 \sum_{\substack{k\in\Znu_2^+ \cup\{0\},\\ |k|\lt N}}
   \bigl(A_k(g) A_k(g) + B_k(g) B_k(g)\bigr)\, [\ffi\circ g]
   = \nu \, \lap \ffi  \circ g.
}
\end{cor}
\begin{proof}
The computation that we made in \eqref{diff_along_vfield}
but applied to $\ffi\circ g$ implies that
\aa{
A_k(g)[\ffi\circ g] = \Bigl[\frac1{|k|^{\al+1}} \cos(k\cdot\te) (\bar k,\nab)\ffi(\te) \Bigr]
\circ g.
}
Similarly, we compute $B_k(g)[\ffi\circ g]$.
Now we just have to repeat the proof of Lemma \ref{lapl} to come to \eqref{lap_smooth_f}.
\end{proof}
\begin{lem}
  \label{lem11}
  Let $\Phi_r$, $r\in [t,T]$, $t\in [0,T)$, be an $H^\al(\Tor,\Rnu^2)$-valued stochastic process
  whose trajectories are integrable, and let $\phi_T$ be an $H^\al(\Tor,\Rnu^2)$-valued
  random element so that both $\Phi_r$ and $\phi_T$ possess finite expectations.
  Then there exists an $\mc F_s$-adapted 
  $H^\s(\Tor,\Rnu^2)\x \mc L\bigl(E,H^\s(\Tor,\Rnu^2)\bigr)$-valued pair 
  of stochastic processes $(Y_s, X_s)$
  solving the BSDE 
  \aaa{
  \label{another_bsde}
   Y_s = \phi_T + \int_s^T \Phi_r\, dr -\int_s^T X_r\, dW_r
  }
  on $[t,T]$. 
  The $Y_s$-part of the solution has the representation
  \aaa{
  \label{Y_solu}
  Y_s = \E \, [ \, \phi_T + \int_s^T \Phi_r \, dr \,| \, \mc F_s ] \,,
  }and therefore is unique. The $X_s$-part of the solution is unique with respect to the norm
  $\|X_s\|^2 = \int_t^T \|X_s\|_{\mc L(E,H^\s(\Tor,\Rnu^2))}^2 \, ds$.
 \end{lem}
  The proof of the lemma uses some ideas from \cite{pardoux90}.
  \begin{proof}
Representation \eqref{Y_solu}  follows from \eqref{another_bsde}.
Let us extend the process $Y_s$ to the entire interval $[0,T]$
by setting $Y_s = Y_t$ for $s\in [0,t]$,
and note that the extended process $Y_s$ is a solution of the SDE
\aa{
 Y_s = \phi_T + \int_s^T \ind_{[t,T]}\Phi_r \, dr - \int_s^T X_rdW_r
 }
  on $[0,T]$.  Let
  $X_s \in \mc L\bigl(E,H^\s(\Tor,\Rnu^2)\bigr)$, $s\in [0,T]$,  be such that 
  \aaa{
  \label{X_solu}
  \E \,\bigl[ \,\phi_T + \int_0^T \ind_{[t,T]}\, \Phi_r\,dr - Y_0 \, | \, \mc F_s \bigr] = 
   \int_0^s X_r \, dW_r.
  }
  The process $X_s$
  exists by the martingale representation theorem.
  Indeed, on the right-hand side of \eqref{X_solu} we have a Hilbert space valued martingale.
 
  By Theorem 6.6 of \cite{Wat},  
  each component of the $H^\al(\Tor,\Rnu^2)$-valued martingale on the right-hand side of \eqref{X_solu}
 can be represented as a sum of real-valued stochastic integrals with respect to the Brownian
 motions $\{\beta^A_k(s),\beta^B_k(s)\}_{k\in \Znu^+_2\cup \{0\}, |k|\lt N}$. Hence, there exist
$\mc F_s$-adapted stochastic processes $\{X_s^{kA}, X_s^{kB}\}_{k\in \Znu^+_2\cup \{0\}, |k|\lt N}$
 such that
  \aa{
 \E \bigl[ \phi_T + \hspace{-2mm} \int_0^T\hspace{-1mm} \ind_{[t,T]} \Phi_r\,dr - Y_0 \, | \, \mc F_s \bigr] = 
 \sum_{\substack{k\in \Znu^+_2\cup \{0\},\\ |k|\lt N}} 
 \int_0^s X_r^{kA}\,d\beta^A_k(r) +\int_0^s X_r^{kB}\,d\beta^B_k(r).
}
Let the process $X_s$ be defined by \eqref{process-x} via
the processes $X_s^{kA}$ and  $X_s^{kB}$, $k\in \Znu^+_2\cup \{0\}, |k|\lt N$.
It\^o's isometry shows that $\E\int_0^T \|X_r\|^2_{\mc L(E,H^\s(\Tor,\Rnu^2))} < \infty$.
Note that for  all $s\in [0,t]$, $\int_0^s X_r dW_r = \int_0^t X_r dW_r$. 
This shows that $X_s=0$ for almost all $\om\in\Om$ and almost all $s\in[0,t]$, and
therefore can be chosen equal to zero on $[0,t]$. Thus, \eqref{X_solu} takes the form:
\aaa{
  \label{X_solu1}
  \E \,\bigl[ \,\phi_T + \int_t^T \Phi_r\,dr - Y_t \, | \, \mc F_s \bigr] = 
   \int_t^s X_r \, dW_r.
  }
It is easy to verify that the pair $(Y_s,X_s)$
  defined by \eqref{Y_solu} and \eqref{X_solu1} solves BSDE \eqref{another_bsde}.
  To prove the uniqueness, note that
  any $\mc F_s$-adapted solution to \eqref{another_bsde}
  takes the form \eqref{Y_solu}, \eqref{X_solu1}.
  Moreover, if the processes $X_s$ and  $X'_s$ 
  satisfy \eqref{X_solu1}, then
  \aa{
  \int_t^T \|X_s - X'_s\|_{\mc L(E,H^\al(\Tor,\Rnu^2))}^2 \, dr =
  \Bigl\|\int_t^T (X_s - X'_s) \, dW_r\Bigr\|_{H^\s(\Tor,\Rnu^2)}^2 = 0.
  }
    \end{proof}
 \begin{proof}[Proof of Theorem \ref{thm2}]
  Let us consider BSDE \eqref{bwd_sde} as an $L_2(\Tor,\Rnu^2)$-valued
  SDE, and 
 $\hat Y_s$  as a function $\Gr \to L_2(\Tor,\Rnu^2)$.
 Since for each $s\in [t,T]$, $y(s,\fdot) \in H^{\s+1}(\Tor,\Rnu^2)$ and
 $\s > 2$ by assumption, then
 $\hat Y_s: \Gr \to L_2(\Tor,\Rnu^2)$ is at least $\C^2$-smooth. 
 Equations \eqref{ns_b} show that the function 
 $\pl_s y(\fdot,\fdot):\, [t,T] \to L_2(\Tor,\Rnu^2)$
 is continuous since $\nab p$,
 $\lap y$, and $(y,\nab y)$ are continuous functions $[t,T] \to L_2(\Tor,\Rnu^2)$
 by Assumption \ref{asn1}. Taking into account that the diffeomorphisms of $\Gr$
 are volume-preserving, we conclude that for each fixed $g\in \Gr$,  $\pl_s \hat Y_s(g): [t,T]\to L_2(\Tor,\Rnu^2)$
 is a continuous function. 
 Hence, $\hat Y_{\bullet}: [t,T] \x \Gr \to L_2(\Tor,\Rnu^2)$ is
 $\C^1$-smooth in $s\in [t,T]$ and $\C^2$-smooth in $g\in \Gr$.
 It\^{o}'s formula is therefore applicable to $\hat Y_s(Z^{t,e}_s)$.
  Below we use the fact that $Z^{t,e}_s$ is a solution to forward SDE \eqref{forward_sde} 
  and the identity
  $\frac{\pl \hat Y_s}{\pl s}(Z^{t,e}_s) = \frac{\pl y(s,\fdot)}{\pl s} \circ Z^{t,e}_s$.
  For the latter derivative we substitute the right-hand
  side of the first equation of \eqref{ns_b}. 
  The notation $\hat X(g)[\hat Y_s(g)]$ 
  (sometimes without square brackets)
  means differentiation
  of the function $\hat Y_s: \, \Gr \to L_2(\Tor,\Rnu^2)$
  along the right-invariant vector field  $\hat X$ on $\Gr$ 
at the point $g\in \Gr$. The same argument as in Remark \ref{remark2} implies that
  $\hat X(g)[\hat Y_s(g)] = \bar \nab_{\hat X}\hat Y_s (g)$.
  Taking into account this argument, we obtain:
 \mmm{
  \label{Y-BSDE-Ito1}
  \hat Y_{s}(Z^{t,e}_{s})
  - \hat h(Z^{t,e}_{T})  =
   -\int_{s}^{T}   \pl_r \hat Y_r(Z^{t,e}_r)\, dr 
  - \int_{s}^{T} dr \,  
  \hat Y_{r}(Z_{r}^{t,e}) [\hat Y_{r}(Z_{r}^{t,e}) ]\\
 - \int_{s}^{T}dr \, {\ts\frac{\epsilon^2}2} \hspace{-4mm}
 \sum_{k\in \Znu_2^+ \cup\{0\}, |k| \lt N}
   \bigl[A_k(Z_{r}^{t,e})A_k(Z_{r}^{t,e})\hat Y_{r}(Z_{r}^{t,e})
   +B_k(Z_{r}^{t,e})B_k(Z_{r}^{t,e})\hat Y_{r}(Z_{r}^{t,e})\bigr] \\
 - \int_{s}^{T}  \epsilon \,
 \sg(Z^{t,e}_r)\hat Y_{r}(Z_{r}^{t,e})\, dW_r.
  }
  Note that
 \aa{
  \hat Y_{r}(Z_{r}^{t,e}) [\hat Y_{r}(Z_{r}^{t,e}) ]
   = [(y(r,\fdot), \nab)y(r,\fdot)]\circ Z^{t,e}_r.
  } 
 Also, let us observe that
  \aa{
  & \frac{\epsilon^2}2 \sum_{k\in \Znu_2^+ \cup\{0\}, |k| \lt N}
   \bigl[A_k(Z_{r}^{t,e})A_k(Z_{r}^{t,e})\hat Y_{r}(Z_{r}^{t,e})
   +B_k(Z_{r}^{t,e})B_k(Z_{r}^{t,e})\hat Y_{r}(Z_{r}^{t,e})\bigr] \\
  &=  \frac{\epsilon^2}2 \sum_{k\in \Znu_2^+ \cup\{0\}, |k| \lt N}
   \bigl[ \bar\nab_{A_k}\bar\nab_{A_k}\hat Y_r(Z^{t,e}_r)+ 
  \bar \nab_{B_k}\bar\nab_{B_k}\hat Y_r(Z^{t,e}_r) \bigr]\\
   & \phantom{\frac{\epsilon^2}2 \sum_{k\in \Znu_2^+ \cup\{0\}, |k| \lt N}
   \bigl[A_k(Z_{r}^{t,e})A_k(Z_{r}^{t,e})\hat Y_{r}(Z_{r}^{t,e})}
   = \nu [\lap y(s,\fdot)] \circ Z^{t,e}_r  
  }
   where 
  the latter equality holds by Lemma \ref{lapl}, and
  $\epsilon > 0$ is chosen so that
  $\frac{\epsilon^2}2\Bigl(1+ \frac12\sum_{k\in\Znu_2^+, |k|\lt N}\frac1{|k|^{2\s}}
  \Bigr) = \nu$. 
 Note that the terms 
 $\bar\nab_{A_k}\bar\nab_{A_k}\hat Y_r(Z^{t,e}_r)$ and  $\bar\nab_{B_k}\bar\nab_{B_k}\hat Y_r(Z^{t,e}_r)$
 are elements of $TG^{\s-1}$, and therefore
 are well defined in $L_2(\Tor,\Rnu^2)$. 
   Continuing \eqref{Y-BSDE-Ito1}, we obtain:
  \mmm{
   \label{Y-BSDE-Ito}
   \hat Y_{s}^{t}(Z^{t,e}_{s})
  - \hat h(Z^{t,e}_{T}) \\ =
  \int_{s}^{T} dr
   \Bigl[\hat V(r,Z^{t,e}_r) + [(y(r,\fdot), \nab)y(r,\fdot)]\circ Z^{t,e}_r
   +\nu [\lap y(r,\fdot)] \circ Z^{t,e}_r \Bigr] \\
    - \int_{s}^{T}[(y(r,\fdot), \nab)y(r,\fdot)]\circ Z^{t,e}_r\, dr
    - \int_{s}^{T} \hspace{-2mm} \nu [\lap y(r,\fdot)] \circ Z^{t,e}_r\, dr\\
     - \int_{s}^{T} \hspace{-2mm} \epsilon\, \sg(Z^{t,e}_r)\hat Y_r(Z_{r}^{t,e})\, dW_r
  =\int_{s}^{T} \hat V(r,Z^{t,e}_r)\, dr
  - \int_{s}^{T}\epsilon\, \sg(Z^{t,e}_r)\hat Y_{r}(Z_{r}^{t,e}) \, dW_r.
 } 
Thus the pair of stochastic processes 
$(\hat Y_s(Z^{t,e}_s),  \epsilon\, \sg(Z^{t,e}_s)\hat Y_{s}(Z_{s}^{t,e}))$
is a solution to BSDE \eqref{bwd_sde} in $L_2(\Tor,\Rnu^2)$. It is $\mc F_s$-adapted
since $Z^{t,e}_s$ is $\mc F_s$-adapted.
By Lemma \ref{lem11}, we know that there exists a unique $\mc F_s$-adapted
solution $(Y^{t,e}_s, X^{t,e}_s)$ to \eqref{bwd_sde} in $H^\s(\Tor,\Rnu^2)$. Clearly, $(Y^{t,e}_s, X^{t,e}_s)$
is also a unique $\mc F_s$-adapted solution
to  \eqref{bwd_sde} in $L_2(\Tor,\Rnu^2)$. Hence, $Y^{t,e}_s = \hat Y_s(Z^{t,e}_s)$
and $\int_t^T\|X^{t,e}_s - \epsilon\, \sg(Z^{t,e}_s)\hat Y_{s}(Z_{s}^{t,e}))\|^2_{\mc L(E, H^\s(\Tor,\Rnu^2))}\, ds = 0$,
and therefore the pair of stochastic processes
$\bigl(\hat Y_s(Z^{t,e}_s), \epsilon\, \sg(Z^{t,e}_s)\hat Y_{s}(Z_{s}^{t,e})\bigr)$
is a unique $\mc F_s$-adapted solution to BSDE \eqref{bwd_sde} in $H^\s(\Tor,\Rnu^2)$.
 The theorem is proved.
 \end{proof}
\section{Some identities involving the Navier--Stokes solution}
 The backward SDE allows us to obtain the representation below
 for the Navier--Stokes solution. Also, it easily implies
 the well-known energy identity for the Navier--Stokes
 equations.
 \subsection{Representation of the Navier--Stokes solution}
 \begin{thm}
 Let $t\in [0,T]$, 
 and let $Z^{t,e}_s$ be the solution to SDE \eqref{forward_sde}
 on $[t,T]$ with the initial condition $Z^{t,e}_t = e$.
 Then the following representation holds for the solution
 $y(t,\fdot)$ to \eqref{ns_consider}.
 \aa{
 y(t,\fdot) = 
 \E\Bigl[\hat h(Z_T^{t,e}) + \int_t^T \nab p(s, \fdot) \circ Z^{t,e}_s ds \Bigr].
 }
 \end{thm}
 \begin{proof}
 Note that $\hat Y_t(Z^{t,e}_t) = y(t,\fdot)$, and $\E[\int_t^T X_r^{t,e} dW_r] = 0$.
 Taking the expectation from the both parts of \eqref{bwd_sde} at time $s=t$
 we obtain the above representation.
 \end{proof}

 \subsection{A simple derivation of the energy identity}
It\^{o}'s formula applied to the squared $L_2(\Tor,\Rnu^2)$-norm of
 $Y^{t,e}_s$ gives:
 \mmm{
 \label{Ito-appl}
 \|Y^{t,e}_s\|_{L_2}^2 = \|\hat h(Z^{t,e}_T)\|^2_{L_2}  \\ +
 2 \int_s^T (Y^{t,e}_r, \hat V(Z^{t,e}_r))_{L_2} dr  - 2\int_s^T (Y^{t,e}_r,X^{t,e}_r dW_r)_{L_2}
  - \int_s^T \|X^{t,e}_s\|_{L_2}^2 dr.
 }
 Using representation \eqref{X} for the process
 $X^{t,e}_s$ we obtain:
 \aa{
 & \|X^{t,e}_s\|_{L_2}^2 = \epsilon^2 \Bigl[
 \sum_{k\in \Znu_2^+\cup \{0\},\, |k| \lt N}
  \|\nab_{\bar A_k} y(s,\fdot)\|_{L_2}^2 + \|\nab_{\bar B_k} y(s,\fdot)\|^2_{L_2}
 \,\Bigr]\\
 & =\epsilon^2 \Bigl[\,\sum_{k\in\Znu_2^+,\, |k| \lt N} \frac1{|k|^{2\s+2}} \,\|(\bar k,\nab y(s,\fdot))\|^2_{L_2}
 +\|\nab y(s,\fdot)\|^2_{L_2}\,\Bigr] \\ 
 &= \epsilon^2 \Bigl[\,{\ts\frac12} \hspace{-4mm}\sum_{k\in\Znu_2^+, \, |k| \lt N} \hspace{-1mm} \frac1{|k|^{2\s+2}} 
  \bigl( \|(\bar k,\nab y(s,\fdot))\|^2_{L_2}+ \|(k,\nab y(s,\fdot))\|^2_{L_2}\bigr)
  +\|\nab y(s,\fdot)\|^2_{L_2}\Bigr] \\ 
 & = \epsilon^2 \Bigl(1 + \frac12\sum_{k\in\Znu_2^+, \,|k| \lt N} \frac1{|k|^{2\s}} \Bigr)\, \|\nab y(s,\fdot)\|^2_{L_2} 
  = 2\,\nu\, \|\nab y(s,\fdot)\|^2_{L_2}.
}

 Taking the expectation in \eqref{Ito-appl} and using the volume-preserving
 property of $Z^{t,e}_s$, we obtain:
\aa{
\|y(s,\fdot)\|_{L_2}^2 + 2\nu \int_s^T  \|\nab y(r,\fdot)\|_{L_2}^2 dr = \|h\|_{L_2}^2.
}

\section{Constructing the solution to the Navier--Stokes equations
 from a solution to the FBSDEs}
 Let us prove now a result which is, in some sense, a converse of Theorem \ref{thm2}.
In this section we consider \eqref{fbsde-torus} as a system of forward and backward SDEs 
in the Hilbert space $H^\al(\Tor,\Rnu^2)$, where $\al\gt 3$. 
As before, let $\hat V (s, Z^{t,e}_s)$ denote $\nab p(s,\fdot)\circ Z^{t,e}_s$,
and let  $\mc F_s$ denote the filtration  $\sg\{W_r, r \in [0,s]\}$.
\begin{thm}
 \label{converse}
Assume, for an $H^{\al+1}$-smooth function $p(s,\fdot)$, $s\in [0,T]$, 
and for any $t\in (0,T)$, the existence of an $\mc F_s$-adapted
solution $(Z^{t,e}_s,Y^{t,e}_s,X^{t,e}_s)$ to \eqref{fbsde-torus} on $[t,T]$
such that the processes  $Z^{t,e}_s$ and $Y^{t,e}_s$ have a.s. continuous 
trajectories and such that $Z^{t,e}_s$ take values in $\Gr$.
Then  there exists  $T_0>0$ such that for all $T<T_0$
there exists a deterministic function $y(s,\fdot) \in T_e\Gr$  on $[0,T]$,  such that a.s. on $[t,T]$
the relation $Y^{t,e}_s = y(s,\fdot) \circ Z^{t,e}_s$ holds. Moreover,
the pair of functions $(y,p)$ solves the backward Navier--Stokes
equations  \eqref{ns_consider} on $[0,T]$.
 \end{thm}
 Lemmas \ref{lem10}--\ref{lem_last} below are the steps in the proof of
 Theorem \ref{converse}.
 \begin{lem}
 \label{lem10}
For all $t\in [0,T)$ and for any $\mc F_t$-measurable $\Gr$-valued random variable $\xi$,
the triple of stochastic processes
\aaa{
 \label{form_of_solution}
 (Z_s^{t,\xi}, Y_s^{t,\xi}, X_s^{t,\xi}) =   
 (Z^{t,e}_s\circ \xi,Y^{t,e}_s\circ \xi, X^{t,e}_s\circ \xi)
 }
is $\mc F_s$-adapted and solves the FBSDEs
 \eee{
 \label{fbsde-xi-t}
 \begin{cases}
  Z_s^{t,\xi} = \xi + \int_t^s Y_r^{t,\xi} \, dr + \int_t^s \sg(Z_r^{t,\xi})\, dW_r \\
  Y_s^{t,\xi} = h(Z_T^{t,\xi}) + \int_s^T \hat V(r,Z_r^{t,\xi}) dr - \int_s^T X_r^{t,\xi} \, dW_r
 \end{cases}
 }
  on the interval $[t,T]$ in the space $H^\al(\Tor,\Rnu^2)$.
  \end{lem}
\begin{proof}
Let us apply the operator $R_\xi$ of the right translation to the both sides
of FBSDEs \eqref{fbsde-torus}.
We only have to prove that we are allowed to write $R_\xi$ under
the signs of both stochastic integrals in \eqref{fbsde-torus}.
Let us prove that it is true for an $\mc F_t$-measurable stepwise function
  $\xi=\sum_{i=1}^\infty g_i\ind_{A_i}$,
where $g_i\in \Gr$ and  the sets $A_i$ are $\mc F_t$-measurable.
Indeed, let $s$ and $S$ be such that
$t\lt s < S \lt T$, and let $\Phi_r$ be an $\mc F_r$-adapted
stochastically integrable process. We obtain:
\mm{
\int_s^S \Phi_r \, dW_r \circ \sum_{i=1}^\infty g_i \ind_{A_i} = 
\sum_{i=1}^\infty \ind_{A_i} \int_s^S \Phi_r \circ g_i \, dW_r =
\sum_{i=1}^\infty \int_s^S \ind_{A_i} \Phi_r \circ g_i \, dW_r \\
= \int_s^S \Phi_r \circ \sum_{i=1}^\infty g_i \ind_{A_i} \, dW_r.
}
Next,  we find a sequence of $\mc F_t$-measurable stepwise functions 
converging to $\xi$ in the space of continuous functions $\C(\Tor,\Rnu^2)$. 
This is possible due to
the separability of $\C(\Tor,\Rnu^2)$. Indeed, let us consider a countable
number of disjoint Borel sets $O^n_i$ covering $\C(\Tor,\Rnu^2)$,
and such that their diameter in the norm of $\C(\Tor,\Rnu^2)$ 
is smaller than $\frac1{n}$. 
Let $A_i^n=\xi^{-1}(O_i^n)$ and $g_i^n \in O_i^n\cap \Gr$.
Define $\xi_n = \sum_{i=1}^\infty g_i^n \ind_{A_i^n}$. 
Then it holds that for all $\om\in\Om$, $\|\xi-\xi_n\|_{\C(\Tor,\Rnu^2)} < \frac1{n}$.
Let $I(\Phi)$ and $I(\Phi\circ \xi)$ denote
$\int_s^S \Phi_r\, dW_r$ and resp. $\int_s^S \Phi_r\circ \xi \, dW_r$.
We have to prove that a.s. $I(\Phi)\circ \xi = I(\Phi\circ \xi)$.
For this it suffices to prove that
\aaa{
\label{61}
\lim_{n\to\infty} \E \|I(\Phi)\circ \xi_n - I(\Phi)\circ \xi\|^2_{L_2(\Tor,\Rnu^2)} &=0, \\
\label{62}
\lim_{n\to\infty} \E \|I(\Phi\circ \xi_n) - I(\Phi\circ \xi)\|^2_{L_2(\Tor,\Rnu^2)} &=0.
}
Due to the volume-preserving property of $\xi$ and $\xi_n$,
$\|I(\Phi)\circ \xi_n\|^2_{L_2(\Tor,\Rnu^2)}=\|I(\Phi)\circ \xi\|^2_{L_2(\Tor,\Rnu^2)}=\|I(\Phi)\|^2_{L_2(\Tor,\Rnu^2)}$.
Hence, by Lebesgue's theorem, in \eqref{61} we can 
pass to the limit under the expectation sign.
 Relation \eqref{61} holds then by
the continuity of $I(\Phi)$ in $\te\in \Tor$. 
To prove \eqref{62} we observe that by It\^o's isometry, the limit in \eqref{62} equals to 
$\lim_{n\to\infty}\E\int_s^S \|\Phi_r\circ \xi_n - \Phi_r \circ \xi\|^2_{L_2(\Tor,\Rnu^2)} dr$.
The same argument that we used to prove \eqref{61} implies that we can pass  to
the limit under the expectation and the integral signs.
Relation \eqref{62} follows from the continuity of $\Phi_r$ in $\te\in\Tor$.

Hence,
 $(Z^{t,e}_s\circ \xi,Y^{t,e}_s\circ \xi, X^{t,e}_s\circ \xi)$
 is a solution to  \eqref{fbsde-xi-t}.
 This solution is clearly $\mc F_s$-adapted.
\end{proof}
Lemmas \ref{lem1}--\ref{y-smooth} below use some ideas and constructions from \cite{Delarue}.
\begin{lem}
 \label{lem1}
 The map $[0,T] \x \Tor \to \Rnu^2$, $(t,\te) \mto Y^{t,e}_t(\te)$  is deterministic.
 \end{lem}
 \begin{proof}
 Let us extend the solution $(Z_s^{t,e}, Y_s^{t,e}, X_s^{t,e})$  to
  the interval $[0,t]$ by setting $Z^{t,e}_s = e$, $Y^{t,e}_s = Y^{t,e}_t$,
  $X^{t,e}_s = 0$ for all $s\in [0,t]$. The extended process
  solves the problem:
  \aaa{
 \label{extended_fbsde}
 \begin{cases}
  Z_s^{t,e} = e + \int_0^s \ind_{[t,T]}(r) Y_r^{t,e} \, dr + \int_0^s \ind_{[t,T]}(r)\sg(Z_r^{t,e})\, dW_r \\
  Y_s^{t,e} = h(Z_T^{t,e}) + \int_s^T \ind_{[t,T]}(r) \hat V(r,Z_r^{t,e}) \, dr - \int_s^T X_r^{t,e} \,dW_r.
 \end{cases}
 }
  The random vector $Y_0^{t,e}$ is $\mc F_0$-measurable, and hence
  is deterministic by Blumenthal's zero-one law.
  Since $Y_t^{t,e} = Y_0^{t,e}$, the result follows. 
 \end{proof}
%
%
\begin{lem} 
\label{lem14}
There exists a constant $T_0 > 0$ such that for $T<T_0$
the function $[0,T] \to H^2(\Tor,\Rnu^2)$, $t\mto Y^{t,e}_t$ is continuous.
\end{lem}
\begin{proof}
Let $(Z^{t,e}_s,Y^{t,e}_s,X^{t,e}_s)$ and $(Z^{t'\!,e}_s,Y^{t'\!,e}_s,X^{t'\!,e}_s)$ be solutions
to \eqref{fbsde-xi-t} which start at the identity $e$ at times $t$ and resp. $t'$, and let $t<t'$.
These solutions can be regarded as solutions of \eqref{extended_fbsde} if we extend them
to the entire interval $[0,T]$ as it was described in Lemma \ref{lem1}.
The application of It\^o's formula to $\|Y^{t,e}_s\|_{L_2(\Tor,\Rnu^2)}^2$ and the backward
SDE of \eqref{fbsde-xi-t} imply that the expectation $\E\|Y^{t,e}_s\|_{L_2(\Tor,\Rnu^2)}^2$ 
is bounded. The forward SDE of \eqref{extended_fbsde}, Gronwall's lemma, and usual 
stochastic integral estimates imply that there exists a constant $K_1>0$ such that
\aa{
\E\|Z^{t,e}_s - Z^{t'\!,e}_s\|^2_{L_2(\Tor,\Rnu^2)} < K_1 \Bigl[
\int_0^s \ind_{[t,T]} \E\|Y^{t,e}_r - Y^{t'\!,e}_r\|^2_{L_2(\Tor,\Rnu^2)} dr + (t'-t)\Bigr].
}
Let us apply It\^o's formula to $\|Y^{t,e}_s - Y^{t'\!,e}_s\|_{L_2(\Tor,\Rnu^2)}^2$
when using the backward SDE of \eqref{extended_fbsde}. Again, 
Gronwall's lemma, usual stochastic integral estimates and the above
estimate for $\E\|Z^{t,e}_s - Z^{t'\!,e}_s\|^2_{L_2(\Tor,\Rnu^2)}$
imply that there exists a constant $K_2 > 0$ such that
\aa{
\E\|Y^{t,e}_s - Y^{t'\!,e}_s\|^2_{L_2(\Tor,\Rnu^2)} < K_2 \Bigl[
\int_0^T \E\|Y^{t,e}_r - Y^{t'\!,e}_r\|^2_{L_2(\Tor,\Rnu^2)} dr + (t'-t)\Bigr].
}
We take $T_0$ smaller than $\frac1{K_2}$. Then
there exists a constant $K>0$ such that
\aaa{
\label{prior-estimate}
\sup_{s\in [0,T]} \E\|Y^{t,e}_s - Y^{t'\!,e}_s\|^2_{L_2(\Tor,\Rnu^2)}
< K(t'-t).
}
Evaluating the right-hand side at the point $s=t$, and taking into
account that $Y^{t'\!,e}_t = Y^{t'\!,e}_{t'}$ 
we obtain that
\aaa{
\label{cont0}
\|Y^{t,e}_t - Y^{t'\!,e}_{t'}\|^2_{L_2(\Tor,\Rnu^2)} < K(t'-t).
}
Differentiating \eqref{extended_fbsde} with respect to $\te$ we 
obtain the following system of forward and backward SDEs:
 \aa{
 \begin{cases}
  \nab Z_s^{t,e} = I + \int_0^s \ind_{[t,T]}(r)\nab Y_r^{t,e} \, dr + \int_0^s \ind_{[t,T]}(r)\nab\sg(Z_r^{t,e})\nab Z_r^{t,e}\, dW_r \\
 \nab Y_s^{t,g} = \nab h(Z_T^{t,e})\nab Z_T^{t,e} + \int_s^T \ind_{[t,T]}(r) \nab \hat V(r,Z_r^{t,g})\nab Z_r^{t,e} \, dr \\
\phantom{\nab Y_s^{t,g} = \nab h(Z_T^{t,e})\nab Z_T^{t,e} } - \int_s^T \nab X_r^{t,e} \,dW_r.
 \end{cases}
 }
Again, standard estimates imply the boundedness 
of $\E\|\nab Z^{t,e}_s\|^2_{L_2(\Tor,\Rnu^2)}$ and $\E\|\nab Y^{t,e}_s\|^2_{L_2(\Tor,\Rnu^2)}$.
The same argument that we used to obtain \eqref{cont0} as well as the estimate
for the $\sup_{s\in [0,T]} \E\|Z^{t,e}_s - Z^{t'\!,e}_s\|^2_{L_2(\Tor,\Rnu^2)}$,
which easily follows from \eqref{prior-estimate}, and the forward SDE
imply that there exists a constant $L>0$ such that for all $t$ and $t'$
from the interval $[0,T]$,
\aaa{
\label{cont1}
\|\nab Y^{t,e}_t - \nab Y^{t'\!,e}_{t'}\|^2_{L_2(\Tor,\Rnu^2)} < L|t'-t|.
}
Differentiating \eqref{extended_fbsde} the second time and using the same argument
once again we obtain that there exist a constant $M>0$ such that
for all $t$ and $t'$ belonging to $[0,T]$, 
\aaa{
\label{cont2}
\|\nab \nab Y^{t,e}_t - \nab \nab Y^{t'\!,e}_{t'}\|^2_{L_2(\Tor,\Rnu^2)} < M|t'-t|.
}
Now \eqref{cont0}, \eqref{cont1}, and \eqref{cont2} imply the continuity 
of the map $t \mto Y^{t,e}_t$ with respect to the $H^2(\Tor,\Rnu^2)$-topology.
\end{proof}
Everywhere below we assume that $T<T_0$ where $T_0$ is the constant defined
in Lemma \ref{lem14}.
\begin{lem}
For every $t\in [0,T)$  and for every $\mc F_t$-measurable
random variable $\xi$, the solution $(Z^{t,\xi}_s, Y^{t,\xi}_s, X^{t,\xi}_s)$ 
to \eqref{fbsde-xi-t} is unique on $[t,T]$.
\end{lem}
\begin{proof}
Let us assume that there exists another solution 
$(\td Z^{t,\xi}_s, \td Y^{t,\xi}_s, \td X^{t,\xi}_s)$ 
to \eqref{fbsde-xi-t} on $[t,T]$. The same argument
as in the proof of Lemma \ref{lem14}  implies the uniqueness
of solution to  \eqref{fbsde-xi-t}. Specifically,  the argument that we applied to
the pair of solutions
$(Z^{t,e}_s,Y^{t,e}_s,X^{t,e}_s)$ and $(Z^{t'\!,e}_s,Y^{t'\!,e}_s,X^{t'\!,e}_s)$
has to be applied to
$(Z^{t,\xi}_s, Y^{t,\xi}_s,  X^{t,\xi}_s)$  and $(\td Z^{t,\xi}_s, \td Y^{t,\xi}_s, \td X^{t,\xi}_s)$, 
and it has to be taken into account that $t=t'$.
\end{proof}
\begin{lem}
\label{lem111}
Let the function $y: [0,T]\x \Tor \to \Rnu^2$ be defined by the formula:
\aaa{
\label{function_y}
y(t,\te) = Y^{t,e}_t(\te).
}
Then, for every $t\in [0,T]$, $y(t,\fdot)$ is $H^\al$-smooth, 
 and a.s.
 \aaa{
 \label{111}
 Y_u^{t,e} = y(u,\fdot)\circ Z_u^{t,e}.
 }
\end{lem}
\begin{proof}
Note that \eqref{form_of_solution} implies that if $\xi$ is $\mc F_t$-measurable
then
\aaa{
\label{NS-rinv}
Y^{t,\xi}_t = y(t,\fdot)\circ \xi.
}
Further, for each fixed $u\in [t,T]$,  $(Z_s^{t,e}, Y_s^{t,e}, X_s^{t,e})$
 is a solution of the following problem on $[u,T]$:
 \[
 \begin{cases}
  Z_s^{t,e} = Z_u^{t,e} + \int_u^s Y^{t,e}_r dr + \int_u^s \sg(Z^{t,e}_r)dW_r \\
  Y_s^{t,e} = h(Z_T^{t,e}) +  \int_s^T \hat V(r,Z_r^{t,e}) dr - \int_s^T X^{t,e}_r dW_r.
 \end{cases}
 \]
  By uniqueness of solution, it holds that 
 $Y_s^{t,e} = Y_s^{u,Z_u^{t,e}}$ a.s. on $[u,T]$. Next, by \eqref{NS-rinv},
 we obtain that $Y_u^{u,Z_u^{t,e}}= y(u,\fdot) \circ Z_u^{t,e}$.
 This implies that there exists a set
 $\Om_u$ (which depends on $u$) of full $\PP$-measure
 such that \eqref{111} holds everywhere on $\Om_u$.
 Clearly, one can find a set $\Om_\Qnu$, $\PP(\Om_\Qnu)=1$, such that
 \eqref{111} holds on $\Om_\Qnu$ for all rational $u\in [t,T]$.
 But the trajectories of $Z^{t,e}_s$ and $Y^{t,e}_s$
 are a.s. continuous. 
Furthermore, Lemma \ref{lem14} implies the continuity 
of $y(t,\fdot)$ in $t$ with respect to (at least) the 
$L_2(\Tor,\Rnu^2)$-topology. 
Therefore, \eqref{111} holds a.s. with respect to the $L_2(\Tor,\Rnu^2)$-topology. Since
 both sides of \eqref{111} are continuous in $\te\in\Tor$ it also holds a.s. for all $\te\in\Tor$.
\end{proof}
\begin{lem}
\label{y-smooth}
The function $y$ defined by formula \eqref{function_y}
is $C^1$-smooth in $t\in [0,T]$.
\end{lem}
\begin{proof}
Let $\dl > 0$. We obtain:
\aa{
y(t+\dl,\fdot) - y(t,\fdot) = Y^{t+\dl,e}_{t+\dl} - Y^{t,e}_t = Y^{t+\dl,e}_{t+\dl} -Y^{t,e}_{t+\dl} +  
Y^{t,e}_{t+\dl} - Y^{t,e}_t.
}
Let $\hat Y_s$ be the right-invariant vector field
on $G^\alpha$ generated by $y(s,\fdot)$. 
Lemma \ref{lem111} implies that a.s.
\aa{
Y^{t,e}_{t+\dl} = \hat Y_{t+\dl}(Z^{t,e}_{t+\dl}).
}
Thus we obtain that a.s.
\aa{
y(t+\dl,\fdot) - y(t,\fdot) = \bigl(\hat Y_{t+\dl}(e) - \hat Y_{t+\dl}(Z^{t,e}_{t+\dl})\bigr)
+(Y^{t,e}_{t+\dl} - Y^{t,e}_t).
}
We use the backward SDE for the second difference and
apply It\^o's formula to the first difference when considering
$\hat Y_{t+\dl}$ as a $C^2$-smooth function $\Gr \to L_2(\Tor,\Rnu^2)$. 
We obtain:
\aa{
& \hat Y_{t+\dl}(Z^{t,e}_{t+\dl}) - \hat Y_{t+\dl}(e) \\
& = \int_t^{t+\dl}\hspace{-2mm} dr \, \hat Y^{t,e}_r(Z^{t,e}_r) [\hat Y_{t+\dl}(Z^{t,e}_r)] 
+ \int_t^{t+\dl} \hspace{-2mm}\epsilon \, \sg(Z^{t,e}_r) \,\hat Y_{t+\dl}(Z^{t,e}_r) \, dW_r\\
&+ \int_t^{t+\dl} dr \sum_{k\in \Znu_2^+\cup \{0\}}  [A_k(Z^{t,e}_r)A_k(Z^{t,e}_r)+B_k(Z^{t,e}_r)B_k(Z^{t,e}_r)]\, 
\hat Y_{t+\dl}(Z^{t,e}_r).
}
The same argument as in Theorem \ref{thm2} implies:
\mm{
\hat Y_{t+\dl}(Z^{t,e}_{t+\dl}) - \hat Y_{t+\dl}(e) =  \int_t^{t+\dl}\hspace{-2mm} dr \, \nab_{y(r,\fdot)}\,y(t+\dl,\fdot)\circ Z^{t,e}_r \\
+\int_t^{t+\dl} \hspace{-2mm} dr \, \nu\,\lap\, y(t+\dl,\fdot)\circ Z^{t,e}_r 
+ \int_t^{t+\dl} \hspace{-2mm}\epsilon \, \sg(Z^{t,e}_r) \,\hat Y_{t+\dl}(Z^{t,e}_r) \, dW_r.
}
Further we have:
\aa{
Y^{t,e}_{t}- Y^{t,e}_{t+\dl} = \int_t^{t+\dl}\hspace{-2mm} dr \, \nab p(r,\fdot) \circ
Z^{t,e}_r - \int_t^{t+\dl} \hspace{-2mm} X^{t,e}_r \, dW_r.
}
Finally we obtain that
\mmm{
\label{diff}
\frac1{\dl}\, \bigl(y(t+\dl,\fdot) - y(t,\fdot)\bigr) = - \frac1{\dl}\, \E \Bigl[
\int_t^{t+\dl}\hspace{-2mm} dr \, [ \, (y(r,\fdot),\nab)\,y(t+\dl,\fdot)\\
+\nu\,\lap\, y(t+\dl,\fdot)+\nab p(r,\fdot)] \circ Z^{t,e}_r
\Bigr].
}
Note that $Z^{t,e}_r$, $\nab p(r,\fdot)$, and 
$(y(r,\fdot),\nab)\,y(t+\dl,\fdot)\circ Z^{t,e}_r$
are continuous in $r$ a.s. with respect to the $L_2(\Tor,\Rnu^2)$-topology.
By Lemma \ref{lem14},
$\nab\, y(t,\fdot)$ and $\lap\, y(t,\fdot)$ are continuous in $t$ with
respect to the $L_2(\Tor,\Rnu^2)$-topology. 
Formula \eqref{diff} and the fact that $Z^{t,e}_t = e$ imply that in the $L_2(\Tor,\Rnu^2)$-topology
\aaa{
\label{ns3}
\pl_t y(t,\fdot) = -[\nab_{y(t,\fdot)}\,y(t,\fdot)
+\nu\,\lap\, y(t,\fdot) + \nab p(t,\fdot)]. 
}
Since the right-hand side of \eqref{ns3} is 
an $H^{\al-2}$-map, so is the left-hand side.
This implies that $\pl_t y(t,\fdot)$ is continuous in $\te\in\Tor$.
Relation \eqref{ns3} is obtained so far for the right derivative of $y(t,\te)$
with respect to $t$.
Note that the right-hand side of \eqref{ns3} is continuous in $t$ which implies that
the right derivative $\pl_t y(t,\te)$ is continuous in $t$ on
$[0,T)$. Hence, it is uniformly continuous on every compact subinterval of $[0,T)$.
This implies the existence of the left derivative of $y(t,\te)$ in $t$, and therefore,
the existence of the continuous derivative $\pl_t y(t,\te)$ everywhere on $[0,T]$.
\end{proof}
\begin{lem}
\label{lem_last}
For every $t\in[0,T]$, the function $y(t,\fdot): \Tor\to\Rnu^2$
is divergence-free. Moreover, the pair $(y,p)$ verifies
the backward Navier--Stokes equations.
\end{lem}
\begin{proof}
Fix a $t>0$, and 
consider the $T_e\Gr$-valued curve $\gm_\zeta = \E\,[{\exp}^{-1}Z^{t,e}_\zeta]$, $\zeta  \gt t$, 
in a neigborhood of the origin of $T_e\Gr$.
The forward SDE of \eqref{fbsde-xi-t} 
can be represented as an SDE on $G^\al$:
\aa{
\begin{cases}
dZ^{t,e}_s = \exp\{\hat Y_s(Z^{t,e}_s)\,ds + \sg(Z^{t,e}_s)\,dW_s\}, \\
Z^{t,e}_t = e,
\end{cases}
}
where $\hat Y_s$ is the right-invariant vector field on $G^\al$ generated
by $y(s,\fdot)$. 
This implies that 
\aa{
\frac{\pl}{\pl\zeta} \gm_\zeta \Bigl |_{\zeta=t}= y(t,\fdot),
}
and therefore $y(t,\fdot)\in T_e\Gr$.
Next, the backward SDE of  \eqref{fbsde-xi-t} implies that $Y^{t,e}_T = h(Z^{t,e}_T)$.
This and relation \eqref{111} imply that $y(T,\fdot) = h$.
Since we already obtained \eqref{ns3} in Lemma
\ref{y-smooth} the proof of the lemma is now complete.
\end{proof}
 \section{The backward SDE as an SDE on a tangent bundle}

 Let $(Z^{t,e}_s, Y^{t,e}_s, X^{t,e}_s)$
 be a solution to FBSDEs \eqref{fbsde-torus}.
 We will show that the backward SDE  can be represented as an SDE on
 the tangent bundle $T\Gr$ as well as an SDE on $TG^\s$.
We will construct a backward SDE in the Dalecky--Belopolskaya form (see \cite{Belopolskaya})
 and show that the process $Y^{t,e}_s$ is its unique solution.

 \subsection{The representation of the backward SDE on $T\Gr$}
  \label{repres_TG}
  Let $y(s,\fdot)$, $s\in [t,T]$, be the solution to the backward Navier--Stokes
  equations \eqref{ns_consider}. 
  Let $\hat Y_s$ be the right-invariant vector field on $\Gr$ generated by
  $y(s,\fdot)$. The connection map on the manifold $\Gr$ generates
  the connection map on the manifold $T\Gr$ as it was shown in \cite{Belopolskaya}, p. 58
  (see also \cite{eliasson}).
  As before, we consider the Levi-Civita connection of the weak Riemannian
  metric \eqref{weak} on $\Gr$.
  Let $\ovl{\exp}$ denote the exponential map of the generated
  connection on $T\Gr$. More precisely, $\ovl{\exp}$ is given
  as follows:
  \aa{
  \ovl{\exp}_{\left( \substack{ x\\ a}\right)} \left( \begin{matrix} \al \\ \beta
  \end{matrix} \right) =  \left( \begin{matrix} \gm_\al(1) \\ \eta_\beta(1)
  \end{matrix} \right)
  }
  where $\left( \begin{matrix} \gm_\al(t) \\ \eta_\beta(t) \end{matrix} \right)$
  is the geodesic curve on $T\Gr$ with the initial data
  $\gm'_\al(0) = \al$, $\eta_\beta'(0) = \beta$, $\gm_\al(0) = x$, $\eta_\beta(0) = a$.
  Let the vector fields
  $A_k^{\sc H}$ and $B_k^{\sc H}$ be the horizontal
  lifts of $A_k$ and $B_k$ onto $TT\Gr$.
  Further let
  $\pl_s \hat Y_s^\ell$ be the vertical lift
  of $\pl_s \hat Y_s$ onto $TT\Gr$.
 Let us consider the backward SDE on $T\Gr$:
 \aaa{
 \label{eqBSDE}
 \begin{split}
  dY^{t,e}_s &= \ovl{\exp}_{Y^{t,e}_s}
  \Bigl \{\pl_s \hat Y_s^\ell(Y^{t,e}_s) ds + {\mathrm S}(Y^{t,e}_s)ds \\
    & +  \epsilon \, \hspace{-2mm} \sum_{k\in \Znu_2^+\cup\{0\}, |k| \lt N}
    \bigl[A_k^{\sc H}(Y^{t,e}_s)\ox e^A_k + B_k^{\sc H}(Y^{t,e}_s)\ox e^A_k\bigr]\, dW_s \Bigr\}, \\
  Y^{t,e}_T & =  \hat h(Z^{t,e}_T)
  \end{split}
  }
  where $\mathrm S$ is the geodesic spray of the Levi-Civita connection
  of the weak Riemannian metric \eqref{weak} on $\Gr$ (see \cite{Gliklikh1} or \cite{Gliklikh2}),
  and $Z^{t,e}_s$, $s\in [t,T]$, is the solution to \eqref{forward_sde} on $\Gr$
  with the initial condition $Z^{t,e}_t = e$. 
  \begin{thm}
  \label{thm8}
  There exists a solution to \eqref{eqBSDE} on $[t,T]$.
  Moreover, if $\pl_s y(s,\fdot)\in H^\s(\Tor,\Rnu^2)$, then
  this solution is unique and
  coincides with the $Y^{t,e}_s$-part of the unique $\mc F_s$-adapted
  solution $(Y^{t,e}_s, X^{t,e}_s)$ to \eqref{bwd_sde}.
  \end{thm}
  \begin{proof}
  From the proof of Theorem \ref{thm2} we know that the pair of stochastic processes
  $(\hat Y_s(Z^{t,e}_s), \epsilon \, \sg(Z^{t,e}_s) \hat Y_s(Z^{t,e}_s))$ is the unique
  $\mc F_s$-adapted solution to \eqref{bwd_sde} in $H^\s(\Tor,\Rnu^2)$.
  Let us prove that $\hat Y_s(Z^{t,e}_s)$ is a strong solution
  to \eqref{eqBSDE}.
  First we describe a system of local coordinates
  $(g^{kA}, X^{kA}, g^{kB}, X^{kB})_{k\in \Znu_2^+ \cup \{0\}}$
  in a neighborhood $U_eg \x T_e\Gr$
  of the point $\hat X(g)\in T\Gr$ where $U_e\sub \Gr$ is
  the canonical chart.
  The vector $\bar g = (g^{kA}, g^{kB})_{k\in \Znu_2^+ \cup \{0\}}$ 
  is the vector of normal coordinates
  in the neighborhood  $U_eg$, $g\in \Gr$.
  The vector $\bar X = (X^{kA}, X^{kB})_{k\in \Znu_2^+ \cup \{0\}}$
  represents the coordinates of the decomposition of the vector
  $\hat X(g) \in T\Gr$ in the basis $\{A_k,B_k\}_{k\in \Znu_2^+ \cup \{0\}}$:
  $\hat X(g) = \sum_{k\in \Znu_2^+ \cup \{0\}}(X^{kA}A_k(g) + X^{kB}B_k(g))$.
    Let $f$ be a smooth function on $T\Gr$, and let
  $\td f(\bar X, \bar g) = f(\hat X(g))$,
  where $\hat X(g) \in T\Gr$.
  %
%
  Let $\tau$ be the exit time of the process $Z^{t,e}_{r}$
  from the neighborhood $U_eZ^{t,e}_s$.
  We will compute the difference $f(Y^{t,e}_s) - f(Y^{t,e}_\tau)$
  using It\^{o}'s formula.  Let
  $(\bar Z_r, \bar Y_r) = (Z^{kA}_r, Z^{kB}_r, Y^{kA}_r, Y^{kB}_r)_{k\in \Znu_2^+ \cup \{0\}}$
  be the vector of local coordinates of the process $\hat Y_r(Z^{t,e}_{r})$ on $[s,\tau]$.
  Using SDE \eqref{eqBSDE}, we obtain:
  \mmm{
  \label{ito5}
  f(Y^{t,e}_s) - f(Y^{t,e}_\tau) = - \sum_{k\in\Znu^+_2 \cup \{0\}} \int_s^\tau
  \Bigl[
   (Y^{kA}_r)'\frac{\pl \td f(\bar Y_r, \bar Z_r)}{\pl Y^{kA}_r}
    + (Y^{kB}_r)'\frac{\pl \td f(\bar Y_r, \bar Z_r)}{\pl Y^{kB}_r} \\
    + Y^{kA}_r \frac{\pl \td f(\bar Y_r, \bar Z_r)}{\pl Z^{kA}_r}
    + Y^{kB}_r \frac{\pl \td f(\bar Y_r, \bar Z_r)}{\pl Z^{kB}_r}
    + \frac{\epsilon^2}2 \dl_k\Bigl(\frac{\pl^2}{\pl (Z^{kA}_r)^2} +
      \frac{\pl^2}{\pl (Z^{kB}_r)^2}\Bigr) \td f(\bar Y_r, \bar Z_r)\Bigr] dr\\
    - \epsilon \hspace{-2mm} \sum_{k\in \Znu_2^+\cup\{0\}, |k| \lt N}
    \int_s^\tau
    \Bigl[\frac{\pl \td f(\bar Y_r, \bar Z_r)}{\pl Z^{kA}_r} \ox e^A_k +
    \frac{\pl \td f(\bar Y_r, \bar Z_r)}{\pl Z^{kB}_r} \ox e^A_k\Bigr]\, dW_r
  }
  where $\dl_k = 1$ if $|k|\lt N$, and $\dl_k = 0$ otherwise.
  Since $f$ is a smooth function on $T\Gr$, all its restrictions
  to the tangent spaces of $\Gr$ are smooth. Hence,
  one can talk about derivatives of $f$ restricted to a tangent space along the vectors
  of this tangent space. Namely, the following relation holds:
  \aa{
  \frac{\pl \td f(\bar Y_r, \bar Z_r)}{\pl Y^{kA}_r} =  f'(\hat Y_r(Z^{t,e}_r)) A_k(Z^{t,e}_r).
  }
  Note that the differentiation of $\td f$ with respect to $Z^{kA}_r$ and $Z^{kB}_r$ can be regarded as 
   the differentiation of the composite function $f\circ \hat Y_r$
     along the vectors $A_k$ and $B_k$. Namely,
     $\frac{\pl \td f(\bar Y_r, \bar Z_r)}{\pl Z^{kA}_r} = A_k(Z^{t,e}_r)[(f\circ \hat Y_r)(Z^{t,e}_r)]$.
  This implies:
   \mmm{
   \label{ito8}
   f(Y^{t,e}_s) - f(\hat h(Z^{t,e}_T)) = -\int_s^T \, dr\,
 \bigl[\, \pl_r (f\circ \hat Y_r)(Z^{t,e}_r)+ \hat Y_r(Z^{t,e}_r)(f\circ \hat Y_r)(Z^{t,e}_r)\\
  +\frac{\epsilon^2}2 \sum_{k\in\Znu^+_2 \cup \{0\}, |k|\lt N} 
  \bigl( A_k(Z^{t,e}_r)A_k(Z^{t,e}_r) +  B_k(Z^{t,e}_r)B_k(Z^{t,e}_r)\bigr)(f\circ \hat Y_r)(Z^{t,e}_r)
  \, \bigr]\\
  -\epsilon \hspace{-2mm}\sum_{k\in\Znu^+_2 \cup \{0\}, |k|\lt N} \int_s^T 
  [ A_k(Z^{t,e}_r)(f\circ \hat Y_r)(Z^{t,e}_r) \ox e^A_k + 
    B_k(Z^{t,e}_r)(f\circ \hat Y_r)(Z^{t,e}_r) \ox e^B_k] dW_r.
   }
 We extended the integration to the entire
 interval $[s,T]$ since the local coordinates no longer appear
 under the integral signs. 
 This is also possible since \eqref{ito5} holds also with respect to 
 the local
 coordinates in the neighborhood $U_eZ^{t,e}_\tau$ and a new
 exit time $\tau_1$. The same argument can be repeated
 with respect to the local coordinates in
 the neighborhood $U_eZ^{t,e}_{\tau_1}$, etc. 
  Let us consider now $f\circ \hat Y_s$ as a time-dependent function of $g\in \Gr$.
  Applying It\^{o}'s formula to $(f\circ \hat Y_s)(Z^{t,e}_s)$
  on the interval $[s,T]$ and
  using SDE \eqref{forward_sde} on $\Gr$,
  we obtain exactly the above identity.
  This proves that $Y^{t,e}_s = \hat Y_s(Z^{t,e}_s)$ is a strong
  solution to \eqref{eqBSDE} on $T\Gr$.
  By results of \cite{Gliklikh1},  $\pl_s \hat Y_s^\ell$ is $\C^1$-smooth.
  Moreover $\mathrm S$, $A_k^{\sc H}$ and  $B_k^{\sc H}$, $k\in \Znu_2^+$, are $\C^\infty$-smooth.
  Again, by results of \cite{Gliklikh1}, the solution of BSDE \eqref{eqBSDE} 
  on $T\Gr$ is unique.
   \end{proof}
  \subsection{The representation of the backward SDE on $TG^\s$}
 Applying Proposition 1.3 (p. 146) of \cite{Belopolskaya} (see also \cite{Gliklikh2}, p. 64) 
 to the manifolds $T\Gr$ and $TG^\s$ and 
 the identical imbedding $\imath_{\sss V}: \, T\Gr \to TG^\s$, we obtain that the process
 $\imath_{\sss V}\bigl(\hat Y_s(Z^{t,e}_s)\bigr)= \hat Y_s(Z^{t,e}_s)$ solves the following
 backward SDE on $TG^\s$:
  \aaa{
 \label{eqBSDE1}
 \begin{split}
  dY^{t,e}_s &= \bar \exp_{Y^{t,e}_s}
  \Bigl \{\pl_s \hat Y_s^{\bar\ell}(Y^{t,e}_s) ds + \bar{\mathrm S}(Y^{t,e}_s)ds \\
    & +  \epsilon \, \hspace{-2mm} \sum_{k\in \Znu_2^+\cup\{0\}, |k| \lt N}
    \bigl[A_k^{\bar{\sc H}}(Y^{t,e}_s)\ox e^A_k + B_k^{\bar{\sc H}}(Y^{t,e}_s)\ox e^A_k\bigr]\, dW_s \Bigr\}, \\
  Y^{t,e}_T & =  \hat h(Z^{t,e}_T)
  \end{split}
  }
 where  $\bar {\mathrm S}$  is the geodesic spray of the Levi-Civita connection
 of the weak Riemannian metric on $G^\s$,
 $\pl_s \hat Y_s^{\bar\ell}$ denotes the vertical lift of
  $\pl_s \hat Y_s$ onto $TTG^\s$, $A_k^{\bar{\sc H}}$ and $B_k^{\bar{\sc H}}$
  denote the horizontal lifts of $A_k$ and $B_k$ onto $TTG^\s$, 
  the process $Z^{t,e}_s$, $s\in [t,T]$, is the solution to \eqref{forward_sde}
  on $G^\s$ with the initial condition $Z^{t,e}_t = e$.
  The exponential map $\bar \exp$ on $TTG^\s$
  is defined similarly to the map $\ovl{\exp}$ on $TT\Gr$. Namely,
  the Levi-Civita connection of the weak Riemannian metric on $G^\s$
  generates a connection on $TG^\s$. The latter
   gives rise to the exponential map $\bar \exp$ on
   $TTG^\s$ as it was described in Paragraph \ref{repres_TG}.
  We actually have obtained the following theorem. 
  \begin{thm}
   Backward SDE \eqref{eqBSDE1} has a unique
   strong solution. Moreover, this solution coincides
   with the unique strong solution to BSDE \eqref{eqBSDE}
   on $T\Gr$, and with the $Y^{t,e}_s$-part of the unique $\mc F_s$-adapted
  solution $(Y^{t,e}_s, X^{t,e}_s)$ to \eqref{bwd_sde}.
  \end{thm}
 \begin{proof}
 We have already shown that the process $\hat Y_s(Z^{t,e}_s)$
 solves BSDE \eqref{eqBSDE1}. The uniqueness of solution
 can be proved in exactly the same way as 
 the uniqueness of solution to \eqref{eqBSDE} on $T\Gr$
 (see the proof of Theorem \ref{thm8}).
  \end{proof}
 \section{Appendix}
\subsection{Geometry of the group of volume-preserving
 diffeomorphisms of the $n$-dimensional torus}
 Let $\mathbb T^n = \underbrace{S^1 \x \dots \x S^1}_n$
 denote the $n$-dimensional torus.
 Let us describe
 a basis of the tangent space $\Alg$ of the group $\Gr$ of
 volume-preserving diffeomorphisms of $\mathbb T^n$.
 We introduce the following notation:
 \aa{
  & \Znu_n^+ = \{(k_1,k_2, \ldots, k_n)\in\Znu^n: k_1 > 0 \; \text{or}
   \; k_1= \cdots = k_{i-1} = 0, \, k_{i}>0, \\
  &\phantom{\Znu_n^+ = (k_1,k_2, \ldots, k_n)\in\Znu^n: k_1 > 0 k_1= \cdots = k_{i-1} = 0, }
    i=2, \dots, n\};  \\
  &k=(k_1,\ldots, k_n)\in \Znu_n^+,  \quad  |k|=\sqrt{\sum_{i=1}^n k_i^2}, \quad
  k\cdot\te = \sum_{i=1}^n k_i\te_i \, ,\\
  &\te = (\te_1, \ldots, \te_n) \in \mathbb T^n, \;
   \nab = \Bigl(\frac{\pl}{\pl\te_1}, \frac{\pl}{\pl\te_2}, \dots, \frac{\pl}{\pl\te_n}\Bigr).
 }
 For every $k\in \Znu_n^+$,
 $(\bar k^1, \ldots, \bar k^{n-1})$ denotes an orthogonal
 system of vectors of length $|k|$
 which is also orthogonal to $k$.
 Introduce the vector fields on $\mathbb T^n$:
 \aa{
 \bar A^i_k = \frac1{|k|^{\s+1}}\cos(k\cdot\te)\, \bar k^i, \;
 \bar B^i_k = \frac1{|k|^{\s+1}}\sin(k\cdot\te)\, \bar k^i, \;
 i=1, \ldots, n-1,\;  k\in \Znu_n^+,
 }
 and the constant vector fields $\bar A^i_0$,
 $i=1, \ldots, n$, whose $i$th
 coordinate is $1$ and the other coordinates are $0$.
 Let $A^i_k, B^i_k$, $i=1, \ldots, n-1$, $k\in \Znu_n^+$,
 denote the right-invariant vector fields on $\Gr$
 generated by  $\bar A^i_k, \bar B^i_k$, $i=1, \ldots, n-1$, $k\in \Znu_n^+$,
 respectively, and let $A^i_0 = \bar A^i_0$, $i=1, \ldots, n$,
 stand for constant vector fields on $\Gr$.
 The following lemma is an analog
 of Lemma \ref{Gsv-hilbert}.
 \begin{lem}
 The vectors $A^i_k(g)$, $B^i_k(g)$, $k\in \Znu_n^+$,
 $i=1, \ldots, n-1$, $g\in\Gr$, $A^i_0$, $i=1, \ldots, n$,
 form an orthogonal basis of
 the tangent space $T_g\Gr$
 with respect to both the weak and the strong inner products
 in $T_g\Gr$.
 In particular, the vectors
 $\bar A^i_k$, $\bar B^i_k$, $k\in \Znu_n^+$,
  $i=1, \ldots, n-1$, $\bar A^i_0$, $i=1, \ldots, n$,
 form an orthogonal basis of the tangent space $\Alg$.
 Moreover, the weak and the strong norms of the basis vectors are bounded
 by the same constant.
 \end{lem}
 The other lemmas of Section \ref{Geometry2D} hold
 in the $n$-dimensional case, with respect
 to the system
 $A^i_k$, $B^i_k$, $k\in \Znu_n^+$,
  $i=1, \ldots, n-1$, $A^i_0$, $i=1, \ldots, n$, without changes.
  The index $\s$ of the Sobolev space $H^\s$  is an integer bigger than
  $\frac{n}2 + 1$.
 \subsection{The Laplacian of a right-invariant vector field on $G^\s(\mathbb T^n)$}
 One of the most important steps in the proof of Theorems \ref{thm2} and \ref{converse}
 is Lemma \ref{lapl}, i.e. the computation of the Laplacian
 of a right-invariant vector field on $G^\s$ with respect to the subsystem
 $\{A_k,B_k\}_{k\in\Znu_2^+\cup\{0\}, |k|\lt N}$ where $N$
 can be fixed arbitrary.
 Below we prove an $n$-dimensional analog of this lemma.
 \begin{lem}
 Let $\hat V$ be the right-invariant vector field on
 $G^{\td \al}(\mathbb T^n)$ generated by an $H^{\td{\al}+2}$-vector field $V$
 on $\mathbb T^n$. Further let $\epsilon > 0$ be such that
 \aa{
  \frac{\epsilon^2}2 \Bigl(1+\frac{n-1}{n}\hspace{-2mm}\sum_{k\in\Znu^+_n, |k|\lt N}\frac1{|k|^{2\s}}\Bigr) = \nu.
 } 
  Then for all $g\in G^{\td\al}$,
  \aa{
   \frac{\epsilon^2}2 \,\Bigl[\sum_{k\in \Znu_n^+, |k| \lt N}
   \sum_{i=1}^{n-1}
   \bigl( \bar\nab_{A_k^i}\bar\nab_{A_k^i}+
   \bar\nab_{B_k^i}\bar\nab_{B_k^i}
   \bigr) +
   \sum_{i=1}^n \bar\nab_{A_0^i}\bar\nab_{A_0^i}\Bigr]\, \hat V(g)
   = \nu \,\lap V  \circ g.
  }
 \end{lem}
 \begin{proof}
 As it was mentioned in the proof of Lemma \ref{lem5'},
 it suffices to consider the case $g=e$.
 We observe that for all $i=1, \ldots, n-1$,
 \aa{
  (\bar k^i,\nab)\cos(k\cdot \te) = -\sin(k\cdot \te)(\bar k^i, k) = 0.
 }
  Similarly, $(\bar k^i,\nab)\sin(k\cdot \te) = 0$.
 Then, for $k\in \Znu_n^+$, $\te\in\mathbb T^n$,
  \aa{
  \sum_{i=1}^{n-1} \bar\nab_{A_k^i}\bar\nab_{A_k^i}\hat V(e)(\te) 
 = \frac1{|k|^{2\s+2}}
  \sum_{i=1}^{n-1}
   \cos(k\cdot \te)(\bar k^i,\nab)\bigl[
   \cos(k\cdot \te)(\bar k^i,\nab)V(\te)\bigr] \\
   =   \frac1{|k|^{2\s+2}} \cos(k\cdot \te)^2
   \sum_{i=1}^{n-1} (\bar k^i,\nab)^2 V(\te)
   = \frac1{|k|^{2\s+2}} \cos(k\cdot \te)^2 (|k|^2\lap - (k,\nab)^2) V(\te).
  }
  The latter equality holds by the identity
  $
  \sum_{i=1}^{n-1} (\bar k^i,\nab)^2  + (k,\nab)^2 = |k|^2\lap
  $
  that follows, in turn, from the fact that the system
  $\bigl\{\frac{\bar k^i}{|k|}, \frac{k}{|k|}\bigr\}$,
  $i=1, \ldots, n-1$, forms an orthonormal basis of $\Rnu^n$.
  Similarly,
  \aa{
\sum_{i=1}^{n-1} \bar\nab_{B_k^i}\bar\nab_{B_k^i}\hat V(e)(\te)= 
  \frac1{|k|^{2\s+2}} \sin(k\cdot \te)^2(|k|^2\lap - (k,\nab)^2) V(\te).
  }
  Hence, for each $k\in\Znu_n^+$,
  \aaa{
  \label{summ}
   \sum_{i=1}^{n-1}
(\bar\nab_{A_k^i}\bar\nab_{A_k^i}+ \bar\nab_{B_k^i}\bar\nab_{B_k^i})\hat V(e)(\te)
  = \frac1{|k|^{2\s+2}} (|k|^2\lap - (k,\nab)^2) V(\te).
  }
 Further we have:
 \mm{
 \sum_{k\in\Znu_n^+, |k|\lt N}  \frac1{|k|^{2\s+2}}(k,\nab)^2
 = \frac12 \sum_{k\in\Znu_n, |k|\lt N}  \frac1{|k|^{2\s+2}}(k,\nab)^2 \\
 = \frac12 \sum_{k\in\Znu_n, |k|\lt N}  \frac1{|k|^{2\s+2}}
  \sum_{i=1}^{n} k_i^2 \pl_i^2 +
   \sum_{k\in\Znu_n, |k|\lt N}  \frac1{|k|^{2\s+2}} \sum_{i\ne j} k_ik_j \pl_i\pl_j
 }
  where $\pl_i = \frac{\pl}{\pl \te_i}$, and due to the factor $\frac12$
  we perform the summation over all $k\in\Znu_n$.
  Clearly, the second sum is zero. To show this,
  we have to specify the way of summation.
  Let us collect in a group the terms $k_i k_j \pl_i\pl_j$ attributed
  to those $k\in \Znu_n$ whose coordinates except the $i$th and the $j$th
  coincide, while the $i$th and the $j$th coordinates
  satisfy the following rules: they are obtained from $k_i$ and $k_j$
  attributed to one of the vectors of the group by means of an arbitrary
  assignment of a sign. This operation specifies four vectors.
  The other four vectors are obtained from the first four vectors
  of the group by means of the permutation of the $i$th and
  the $j$th coordinates.
  In total, we get eight vectors in the group.
  Clearly, the summands $k_i k_j \pl_i\pl_j$ attributed to these vectors cancel each other.
  Let us compute the first sum.
  \aa{
  \sum_{k\in\Znu_n, |k|\lt N}  \frac1{|k|^{2\s+2}}
  \sum_{i=1}^{n} k_i^2 \pl_i^2
  = \sum_{i=1}^{n}\, \Bigl[\sum_{k\in \Znu_n, |k|\lt N} \frac1{|k|^{2\s+2}}
  \, k_i^2 \Bigr] \, \pl_i^2.
  }
  Note that
  \aa{
  \sum_{k\in \Znu_n, |k|=const}
  \, k_1^2 = \cdots = \sum_{k\in \Znu_n, |k|=const}
  \, k_n^2 =  \frac1{n} \sum_{k\in \Znu_n, |k|=const}|k|^2.
  }
  This implies:
  \aa{
  \sum_{k\in\Znu_n, |k|\lt N}  \frac1{|k|^{2\s+2}}
  \sum_{i=1}^{n} k_i^2 \pl_i^2 = \frac1{n}
  \sum_{k\in\Znu_n, |k|\lt N} \frac1{|k|^{2\s}}\, \lap
  = \frac2{n} \sum_{k\in\Znu_n^+, |k|\lt N} \frac1{|k|^{2\s}}\, \lap.
  }
  Together with \eqref{summ} it gives:
  \mm{
  \sum_{k\in\Znu_n^+,|k|\lt N}
  \sum_{i=1}^{n-1}
(\bar\nab_{A_k^i}\bar\nab_{A_k^i}+ \bar\nab_{B_k^i}\bar\nab_{B_k^i})\hat V(e)(\te)
 =\frac{n-1}{n}\sum_{k\in\Znu_n^+, |k|\lt N}\frac1{|k|^{2\s}}\, \lap  V(\te).
  }
  We also have to take into consideration the term
  \aa{
  \sum_{i=1}^{n} \bar\nab_{A_0^i}\bar\nab_{A_0^i}\hat V(e)(\te) = \lap V(\te).
  }
  Finally, we obtain:
  \aa{
  \Bigl[
  \sum_{k\in\Znu_n^+, |k|\lt N} \sum_{i=1}^{n-1}
(\bar\nab_{A_k^i}\bar\nab_{A_k^i}+ \bar\nab_{B_k^i}\bar\nab_{B_k^i})
  + \sum_{i=1}^{n}
\bar\nab_{A_0^i}\bar\nab_{A_0^i}
\Bigr]\,
\hat V(e)(\te)  \\
  = \Bigl(1 +
  \frac{n-1}{n}\sum_{k\in\Znu_n^+, |k|\lt N}\frac1{|k|^{2\s}} \Bigr)
  \,\lap V(\te).
  }
  The lemma is proved.
 \end{proof}

\section*{Acknowledgements}
We would like to thank the referee for meaningful questions.
This work was supported by the Portuguese Foundation for Science
and Technology (FCT) under the projects PTDC/MAT/69635/06
and SFRH/BPD/48714/2008.

\end{document}